\documentclass{patmorin}
\usepackage{amsmath}
\usepackage{amsfonts}
\usepackage{amsthm}
\usepackage{graphicx}

\newcommand{\Vspace}[1]{}
\usepackage{enumerate}
\usepackage{cite}
\usepackage{pat}
\usepackage{paralist}
\usepackage{hyperref}
\hypersetup{colorlinks=true, linkcolor=linkblue,  anchorcolor=linkblue,
citecolor=linkblue, filecolor=linkblue, menucolor=linkblue,
urlcolor=linkblue, pdfcreator=Me, pdfproducer=Me} 
\usepackage{algorithm}
\usepackage{subfig}
\usepackage{array}
\usepackage[noend]{algpseudocode}
\usepackage[usenames]{xcolor}
\usepackage{stmaryrd} 
\usepackage{mathtools} 
\usepackage{todonotes}
\usepackage{lineno}

\newif\ifSODA
\SODAfalse

\newtheorem{claimx}{Claim}{\bfseries}{\itshape}

\newcommand{\Fary}{F\'ary}

\newcommand{\xxx}{\includegraphics[width=.9ex]{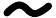}}
\newcommand{\yyy}{\includegraphics[width=.9ex]{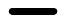}}
\newcommand{\straight}[1]{\stackrel{\yyy}{#1}}
\newcommand{\squiggle}[1]{\stackrel{\xxx}{#1}}

\DeclareMathOperator{\tv}{\squiggle{\mathit{v}}}
\DeclareMathOperator{\sv}{\straight{\mathit{v}}}

\title{\MakeUppercase{Every Collinear Set in a Planar Graph Is Free}\thanks{%
    The work of VD and PM was partly funded by NSERC.
    The work of FF was 
partially supported by MIUR Project “MODE” under PRIN 20157EFM5C and by 
H2020-MSCA-RISE project 734922, “CONNECT”.
    The work of DG  was 
partly funded by the ANR project GATO, under contract
     ANR-16-CE40-0009.}}

\author{Vida Dujmovi\'c,\thanks{Department of Computer Science and Electrical Engineering, University of Ottawa}\quad
        Fabrizio Frati,\thanks{Dipartimento di Ingegneria, Universit\'a Roma Tre}\quad 
        Daniel Gon\c{c}alves,\thanks{LIRMM, Université de Montpellier, CNRS}\quad
        Pat Morin,\thanks{School of Computer Science, Carleton University}\quad 
        and G\"unter Rote\thanks{Institut f\"ur Informatik, Freie Universit\"at Berlin}}

\begin{document}
\begin{titlepage}
\maketitle

\begin{abstract}
  We show that if a planar graph $G$ has a plane straight-line drawing in
  which a subset $S$ of its vertices are collinear, then for any set of
  points, $X$, in the plane with $|X|=|S|$, there is a plane straight-line
  drawing of $G$ in which the vertices in $S$ are mapped to the points
  in $X$.  This solves an open problem posed by Ravsky and Verbitsky in
  2008.  In their terminology, we show that every collinear set is free.

  This result has applications in graph drawing, including untangling,
  column planarity, universal point subsets, and partial simultaneous
  drawings.
\end{abstract}
\end{titlepage}

\tableofcontents

\newpage
\pagenumbering{arabic}


\section{Introduction}


A \emph{straight-line drawing} of a graph $G$ maps each vertex to a point in the plane and each edge to a line segment between its endpoints. A straight-line
drawing is \emph{plane} if no pair of edges cross except at a
common endpoint. A set
of vertices  $S\subseteq V(G)$ in a planar graph $G$ is a \emph{free
  set} if for any set of points $X$ in the plane with $|X|=|S|$, $G$ has a plane
straight-line drawing in which the vertices of $S$ are mapped to the points in $X$.  Free sets are useful tools in graph drawing
and related areas and have been used to settle problems in
untangling~\cite{bose.dujmovic.ea:polynomial,dalozzo.dujmovic.ea:drawing,dujmovic:utility,ravsky.verbitsky:on
}, column planarity~\cite{dalozzo.dujmovic.ea:drawing,dujmovic:utility}, universal point subsets~\cite{dalozzo.dujmovic.ea:drawing,dujmovic:utility},
and partial simultaneous geometric drawings~\cite{dujmovic:utility}.

\begin{figure*}[htb]
  \centering
  \ifSODA
  \includegraphics[scale=0.8]{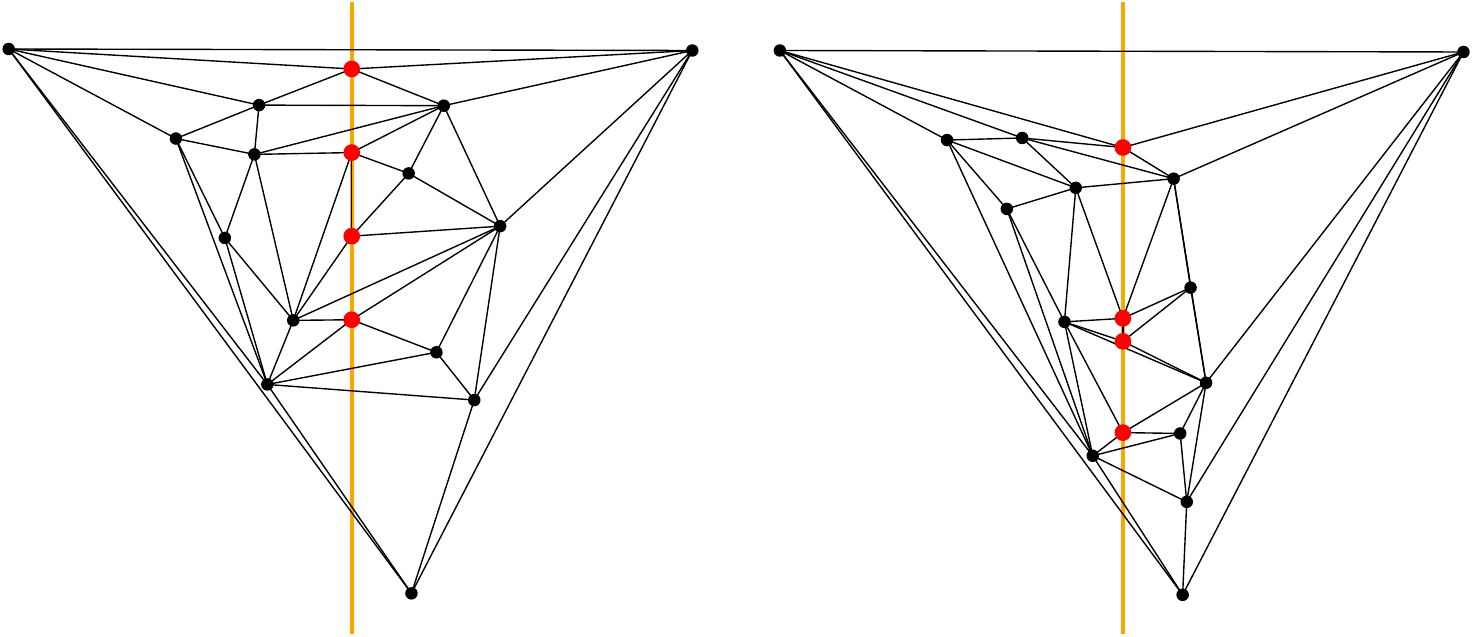}
  \else
  \includegraphics{figs/collinear-example}
  \fi
  \caption{The 4 red vertices form a collinear set $S$. On the
    right, the graph is redrawn so that vertices of $S$ lie at some
    other collinear locations.}
  \label{fig:collinear}
\end{figure*}

 A set of vertices  $S\subseteq V(G)$ in a planar graph $G$ is a
 \emph{collinear set} if $G$ has a plane straight-line drawing in
 which all vertices in $S$ are mapped to a single line,
see \figref{collinear}.
 A collinear set $S$
is a \emph{free collinear set} if, for any collinear set of points in
the plane $X$ with $|X|=|S|$, $G$ has a plane straight-line drawing in
which the vertices of $S$ are mapped to the points in $X$.  
Ravsky and Verbitsky~\cite{ravsky.verbitsky:on
}
define $\sv(G)$ and $\tv(G)$ as the respective sizes of the
largest collinear set and largest free collinear set in $G$, and ask
the following question:
\begin{quote}
	How far or close are parameters $\tv(G)$ and $\sv(G)$? It
	seems that \emph{a priori} we even cannot exclude equality. To clarify
	this question, it would be helpful to (dis)prove that every collinear
	set in any straight-line drawing is free.
\end{quote}
Here, we answer this question by proving that, for every planar graph $G$,
$\tv(G)=\sv(G)$, that is:

\begin{thm}\thmlabel{our-bang}
Every collinear set is a free collinear set. 
\end{thm}

Let $v(G)$ denote the largest free set for a planar graph $G$. Clearly, we have $v(G)\leq \tv(G) \leq \sv(G)$. Further, as discussed in detail below, it is well-known that $v(G)=\tv(G)$. However, prior to our work, the best known bound between $v(G)$,
$\tv(G)$, and $\sv(G)$ in the other direction was $v(G),\tv(G) \geq \sqrt{\sv(G)}$, proved by Ravsky and Verbitsky~\cite{ravsky.verbitsky:on}. 
Thanks to \thmref{our-bang}, we now know a stronger bound, in fact the ultimate $v(G)=
\tv(G) = \sv(G)$ relationship. This relationship was
previously only known for planar $3$-trees
\cite{dalozzo.dujmovic.ea:drawing}. \thmref{our-bang}, in fact, implies a stronger result than $v(G)= \tv(G) = \sv(G)$:



\begin{cor}\corlabel{our-all}
In a planar graph $G$, a set $S\subseteq V(G)$ is a free set if
and only if it is a \mbox{collinear set}.
\end{cor}

That every free set is a collinear set is immediate. \thmref{our-bang} then implies \corref{our-all} since every free collinear set is also a free set. 
This fact, which implies that $v(G)=\tv(G)$, has been observed by several
authors~\cite{bose.dujmovic.ea:polynomial,dalozzo.dujmovic.ea:drawing,dujmovic:utility,gkossw-upg-09}. To
see it,
let $X=\{(x_1,y_1),\ldots,(x_{|S|},y_{|S|})\}$ be the desired target locations on
which
$S$ is supposed to be drawn. By rotation, we may assume that
no two points have the same y-coordinate. Let
$X_0=\{(0,y_1),\ldots,(0,y_{|S|})\}$.  By the definition of free
collinear set, $G$ has a plane straight-line drawing $\Gamma_0$ in
which $S$ maps to $X_0$.  Since the set of plane straight-line drawings of
$G$ is an open set, we can arbitrarily
perturb the vertices in some small neighborhood,
resulting in some
plane straight-line drawing $\Gamma_{\epsilon}$ in which $S$ maps to
$X_\epsilon=\{(\epsilon x_1,y_1),\ldots,(\epsilon x_{|S|},y_{|S|})\}$,
for some $\epsilon >0$.
Dividing all the $x$-coordinates of $\Gamma_\epsilon$ by $\epsilon$ then
yields a plane straight-line drawing 
in which $S$ maps to~$X$.

\begin{figure}[tb]
  \centering
  \ifSODA
  \includegraphics[scale=0.8,page=2]{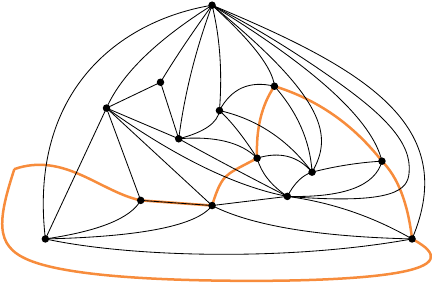}
  \else
  \includegraphics[page=2]{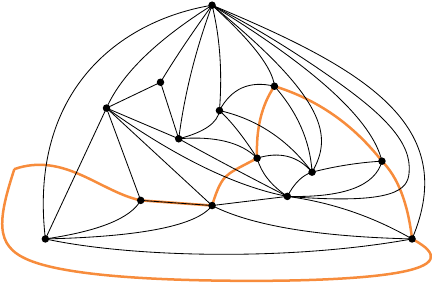}
  \fi
  \caption{A proper good curve}
  \label{fig:proper-good}
\end{figure}

Thus, \thmref{our-bang} is our main result and this paper is dedicated to
proving it. The following
characterization of collinear sets by Da Lozzo,
Dujmovi\'c, Frati, Mchedlidze, and Roselli
~\cite{dalozzo.dujmovic.ea:drawing}  is helpful in that goal.

\begin{defn}
\label{proper-good}
  Given a drawing $G$,
  a
  Jordan curve $C$ is a \emph{proper good curve} 
  if it contains a point in the outer face
of $G$ and the intersection between $C$ and each edge $e$ of $G$ is
either empty, a single point, or the entire edge $e$.  
See \figref{proper-good} for an example.
\textup(This is a conjunction of the two definitions of
\emph{proper} and of \emph{good curves}
from~\cite{dalozzo.dujmovic.ea:drawing}.\textup)
\end{defn}

\begin{thm}\cite{dalozzo.dujmovic.ea:drawing} \thmlabel{collinear-set}
  A set $S$ of vertices of a graph $G$ is a collinear set if and
  only if there is a plane drawing of $G$ and a proper good curve $C$
  that contains every vertex in~$S$.
\end{thm}

The \emph{if} part of this theorem says, in other words, that the
curve $C$ and the edges of $G$ can be simultaneouly straightened
 (after cutting $C$ open at
some point in the outer face)
while keeping the vertices of $S$ on~$C$.
 \thmref{collinear-set} is helpful because it reduces the problem of
finding large collinear sets in a graph $G$ to a topological game in
which one only needs to find a curve that contains many vertices
of $G$.  Da Lozzo \etal~\cite{dalozzo.dujmovic.ea:drawing}
 used \thmref{collinear-set} to give
tight lower bounds on the sizes of collinear sets in planar graphs
of treewidth at most 3 and triconnected cubic planar graphs. Despite the conceptual simplification provided by \thmref{collinear-set},
the identification of collinear sets is highly non-trivial:
Mchedlidze, Radermacher, and Rutter~
\cite{mchedlidze.radermacher.ea:aligned} showed that it is NP-hard to
determine if a given set of vertices in a planar graph is a collinear
set.
Nevertheless, \thmref{collinear-set} is a useful tool for finding large 
collinear sets. In combination with \corref{our-all}, it gives a
characterization of free sets:
\begin{cor}
	A set $S$ of vertices of a graph $G$ is a free set if and
	only if there is a plane drawing of $G$ and a proper good curve $C$
	that contains every vertex of $S$.
      \end{cor}
This is a useful
tool for finding free sets, which have a wide variety of applications,
as outlined in the next section.

\subsection{Applications and Related Work}

The applicability of \corref{our-all} comes from the fact that a number of graph drawing applications require (large) free sets, whereas finding large collinear sets
is an easier task. Indeed there are planar graphs for which large collinear sets were known to exist, however large free sets were not. Those include triconnected cubic planar graphs
and planar graphs of treewidth at least $k$.
%
We now review applications of our result. 



\paragraph{Untangling.}  Given a straight-line drawing of a planar
graph $G$, possibly with crossings, to \emph{untangle} it means to assign
new locations to some of the vertices of $G$ so that the resulting
straight-line drawing of $G$ becomes noncrossing. The goal is to do so while
\emph{keeping fixed} the location of
as many vertices as possible. 

In 1998, Watanabe asked if every polygon can be untangled while keeping at least $\varepsilon n$ vertices
fixed, for some $\varepsilon >0$. Pach and Tardos\cite{pt-up-02} answered that question in
the negative by providing an $\mathcal{O}((n\log n)^{2/3})$ upper bound on the
number of fixed vertices. This has almost been  matched by
an 
$\Omega(n^{2/3})$ lower bound by Cibulka~\cite{c-upg-10}. Several papers have studied the untangling
problem~\cite{pt-up-02,cano.toth.ea:upper,c-upg-10,bose.dujmovic.ea:polynomial,gkossw-upg-09, kpr-upg-11,ravsky.verbitsky:on}. Asymptotically tight
bounds are known for paths \cite{c-upg-10}, trees \cite{gkossw-upg-09}, outerplanar graphs
\cite{gkossw-upg-09}, and planar graphs of treewidth two and three \cite{ravsky.verbitsky:on,
  dalozzo.dujmovic.ea:drawing}. 
  
For general
planar graphs there is still a large gap. Namely, it is known that every planar graph can be untangled while
keeping $\Omega(n^{0.25})$ vertices fixed
\cite{bose.dujmovic.ea:polynomial} (this answered a 
question by Pach and Tardos \cite{pt-up-02})  and that there are planar graphs
that cannot be untangled while keeping $\Omega(n^{0.4948})$ vertices
fixed \cite{cano.toth.ea:upper}. \thmref{our-bang} can help close this gap, whenever a good bound
on collinear sets is known.  
Namely, Bose \etal\cite{bose.dujmovic.ea:polynomial} (implicitly) and  Ravsky and Verbitsky
\cite{ravsky.verbitsky:on} (explicitly) proved that every straight-line
drawing of a planar graph $G$ can be untangled while keeping
$\Omega(\sqrt{|S|})$ vertices fixed, where $S$ is a free set of
$G$. Together with \corref{our-all} this implies that, for untangling, it is enough to
find large collinear sets.

\begin{thm}\thmlabel{our-untang}
Let $S$ be a collinear set of a planar graph $G$. Every straight-line drawing of $G$ can be untangled while keeping $\Omega(\sqrt{|S|})$ vertices fixed.
\end{thm}

Da Lozzo,
Dujmovi\'c, Frati, Mchedlidze, and Roselli
~\cite{dalozzo.dujmovic.ea:drawing}  proved that
every triconnected cubic planar graph has a collinear set of size
$\Omega(n)$. Then \thmref{our-untang} implies the following new result,
for which $\Omega(n^{0.25})$ was a previously best known \mbox{untangling
bound.}

\begin{cor}\corlabel{our-cubic-unt}
Every straight-line drawing of any $n$-vertex triconnected cubic planar
graph can be untangled while keeping $\Omega(\sqrt{n})$ vertices fixed.
\end{cor}

\corref{our-cubic-unt} is almost tight due to the $\mathcal{O}(\sqrt{n\log^3n })$ upper bound for triconnected cubic planar graphs of diameter $\mathcal{O}(\log n)$ \cite{c-upg-10}. \corref{our-cubic-unt} cannot be extended to all bounded-degree planar graphs, see \cite{dujmovic:utility,DBLP:journals/dm/Owens81} for
reasons why.  Da Lozzo \etal~\cite{dalozzo.dujmovic.ea:drawing} also proved that planar graphs of treewidth at least
$k$ have $\Omega(k^2)$-size collinear sets. Together with
\thmref{our-untang}, this implies that 
%
every straight-line drawing of an $n$-vertex planar graph of treewidth
at least $k$ can be untangled while keeping $\Omega(k)$ vertices fixed. 
%
This gives, for example, a tight $\Theta(\sqrt{n})$
untangling bound for planar graphs of treewidth
$\Theta(\sqrt{n})$.


 \paragraph{Universal Point Subsets.}


Closing the gap between $\Omega(n)$ and $\mathcal{O}(n^2)$ on the size of the
smallest \emph{universal point set} (a set of points on which
every $n$-vertex planar graph can be drawn with straight edges by using $n$ of these
points as locations for the vertices) is a major, extensively studied, and difficult graph
drawing problem, open since
$1988$~\cite{
  dFPP90, DBLP:journals/ipl/Kurowski04, DBLP:journals/jgaa/BannisterCDE14}. 
  
The interest universal point sets motivated the following notion introduced by Angelini~\etal~\cite{abehlmmo-ups-12}. 
A \emph{universal point subset} for  a set $\mathcal{G}$ of $n$-vertex planar graphs is a
set $P$ of $k\leq n$ points in the plane such that, for every
$G\in\mathcal{G}$, there is a plane straight-line
drawing of $G$ in which $k$ vertices of $G$ are mapped to the $k$
points in $P$. Every set of $n$ points in general position is a
universal point subset for $n$-vertex outerplanar graphs
\cite{GMPP,DBLP:journals/comgeo/Bose02,DBLP:conf/cccg/CastanedaU96};  every
set of $\lceil \frac{n-3}{8}\rceil$ points in the plane is a universal
point subset for the $n$-vertex planar graphs of treewidth at most
three \cite{dalozzo.dujmovic.ea:drawing}; and, every set of $\sqrt{\frac{n}{2}}$ points in the plane is a universal point
subset for the $n$-vertex planar graphs \cite{dujmovic:utility}. 

Dujmovi\'c~\cite{dujmovic:utility}
  proved that every set of $v(G)$ points in the plane is a universal point subset
  for a planar graph $G$. Together with \corref{our-all} this implies
  that, in order to find large universal point subsets, it is enough to look for large collinear sets.

\begin{thm}\thmlabel{our-subset}
Let $S$ be a collinear set for a graph $G$. Then every set of $|S|$ points in the
plane is a universal point subset for $G$.
\end{thm}

As was the case with untangling, \thmref{our-subset} implies new results
for universal point subsets of triconnected cubic planar graphs and
treewidth-$k$ planar graphs. In particular, \thmref{our-subset} and the
fact that every triconnected cubic planar graph has a collinear set of
size $\ceil{\frac{n}{4}}$ \cite{dalozzo.dujmovic.ea:drawing} imply the
following asymptotically tight result. The previously best known bound
was $\Omega(\sqrt{n})$ \cite{dujmovic:utility}.

\begin{cor}\corlabel{our-cubic-sub}
Every set of $\ceil{\frac{n}{4}}$ points in the plane  is a universal
point subset for  every $n$-vertex triconnected cubic
planar graph.
\end{cor}

Similarly, \thmref{our-subset} and the fact that planar graphs of
treewidth at least $k$ have collinear sets of size $ck^2$, for some
constant $c$ \cite{dalozzo.dujmovic.ea:drawing}, imply that every set
of  $c k^2$ points in the plane is a universal point subset for  such
graphs. This gives, for example, an asymptotically tight  $\Theta(n)$
result on the size of the largest universal point subset for planar
graphs of treewidth $\Theta(\sqrt{n})$. 

For similar applications of \thmref{our-bang}
and \corref{our-all}, such as \emph{column
planarity}~\cite{behks-cppsge-17,dalozzo.dujmovic.ea:drawing,dujmovic:utility}
and \emph{partial simultaneous geometric embeddings with and without
mappings}~\cite{behks-cppsge-17,ddlmw-pqp-15,dujmovic:utility} see a
survey by Dujmovi\'c~\cite{dujmovic:utility}.


\subsection{Proof Outline for \thmref{our-bang}}

We assume w.l.o.g.\ that $G$ is a plane straight-line
drawing in which the collinear set $S\subseteq V(G)$ lies
on the $y$-axis $Y=\{(0,y):y\in\R\}$. Let
$L=\{(x,y)\in\R^2:x<0\}$ and $R=\{(x,y)\in\R^2: x >0\}$ denote the open
halfplanes to the left and right of $Y$.
We consider the points on the $y$-axis $Y$ as being ordered,
with $(0,a)$ before $(0,b)$ if $a<b$. 
We
assume, furthermore, that we are given $|S|$ distinct $y$-coordinates,
and the goal is to find another plane straight-line drawing of $G$ in
which the vertices in $S$ are positioned on $Y$ with the given $y$-coordinates.


The difficulty comes from edges of $G$ that cross $Y$.
These edges must cross $Y$ in prescribed
intervals between the prescribed locations of vertices in $S$, and
these intervals may be arbitrarily small.  An extreme version of this
subproblem is the one in which $G$ is a drawing where every
edge intersects $Y$ in exactly one point (possibly an endpoint) and
the location of each crossing point is prescribed.  The most difficult
instances occur when $G$ is edge-maximal.

In \secref{quadrangulations} we describe these edge-maximal graphs, which
we call A-graphs.  A-graphs are a generalization of quadrangulations, in
which every face is either a quadrangle whose every edge intersects $Y$ or a triangle with one vertex in
each of $L$, $Y$, and $R$.  \thmref{a-graph} 
shows that it
is possible to find a plane straight-line drawing of any A-graph where
the intersections of the drawing with $Y$ occur at prescribed locations.
For this purpose, we set up a system of linear equations and show that
it has a unique solution. This proof involves linear algebra and
 continuity arguments.

In \secref{triangulations} we prove that every collinear set is free.
The technical statement of this result, \thmref{main}, shows a somewhat
stronger result for triangulations that makes it possible not only to
prescribe the locations of vertices on $Y$ but also to nearly prescribe the
points at which edges of the triangulation cross $Y$.  This proof uses combinatorial
reductions that are applied to a triangulation $T$ that either reduce its
size or increase the number of edges that cross $Y$.  When none of these
reductions is applicable to $T$, removing the edges of $T$ that do not
cross $Y$ creates an A-graph, $G$, on which we can apply \thmref{a-graph}.

\secref{definitions} begins our discussion with definitions 
 that we use throughout.


\section{
  Preliminaries}
\seclabel{definitions}

We recall some 
standard definitions. 


%
\ifSODA
\else
A \emph{curve} $C$ is a continuous function from $[0,1]$
to $\mathbb{R}^2$.  The points $C(0)$ and $C(1)$ are the \emph{endpoints} of $C$.  A curve $C$ is \emph{simple} if $C(s)\neq C(t)$
for $s\ne t$ 
except possibly for $s=0$ and $t=1$; it is \emph{closed} if $C(0)=C(1)$.  A \emph{Jordan
	curve} $C\colon [0,1]\to\R^2$ is a simple closed curve.  
We will often not distinguish between a curve $C$ and its
image $\{C(t):0\le t\le 1\}$. 
The \emph{open curve} is the set $\{C(t):0< t< 1\}$.
A point $x\in\R^2$ lies \emph{on} $C$ if $x\in C$.  
For any Jordan curve $C$, $\R^2\setminus C$ has two connected
components: One of these, $C^-$, is finite (the {\em interior} of $C$) and the other, $C^+$, is
infinite (the {\em exterior} of $C$).  
%


\fi

All graphs considered in this paper are finite and simple.   
We use $V(G)$ and $E(G)$ to denote the vertex set and edge
set of $G$, respectively.
We use $xy$
to denote the edge between the vertices 
$x$ and $y$. 

A \emph{drawing} 
of a graph $G$ consists of $G$ together with
 a one-to-one mapping $\varphi\colon V(G)\to\R^2$ and a mapping $\rho$ from
$E(G)$ to curves in $\R^2$ such that, for each $xy\in E(G)$, $\rho(xy)$
has endpoints $\varphi(x)$ and $\varphi(y)$.
We will not distinguish between a drawing $G$ and the underlying graph
$G$, and we will never
explicitly refer to $\varphi$ and $\rho$.
In particular, we will sometimes have a drawing $G$ and we will speak about constructing
a different drawing of $G$, without danger of confusion.
%
A drawing is \emph{straight-line} if each edge is a straight-line
segment. A drawing is \emph{plane} if each edge is a simple curve, and no
two edges intersect, except possibly at common endpoints. A
\emph{F\'ary drawing} is a plane straight-line drawing.
A \emph{plane straight-line graph} is a planar graph $G$ along with an associated \Fary\ drawing of $G$.

By default, an edge curve includes
its endpoints, otherwise we refer to it as an \emph{open} edge.
The \emph{faces} of a plane drawing $G$ are the maximal connected
subsets of $\R^2\setminus\bigcup_{xy\in E(G)} xy$.  One of
these faces, the \emph{outer} face, is unbounded; the other faces are called \emph{inner} or \emph{bounded} faces. A {\em boundary vertex} is incident to the outer face, other vertices are called \emph{interior} vertices.
 If
 $C$ is a
cycle 
 in a plane drawing, then
 there is a 
Jordan curve whose image
is the union of edges in $C$.  In this case, the interior and exterior of
$C$ refer to the interior and exterior of the corresponding Jordan curve.

A \emph{triangulation} (a \emph{quadrangulation}) is a plane drawing,
not necessarily with straight edges,
in which each face is bounded by a 3-cycle (respectively, a 4-cycle). 

A \emph{separating triangle} of a graph $G$ is
a cycle of length 3
 whose removal disconnects $G$.

The \emph{contraction} of an edge $xy$ in a graph $G$ identifies $x$
and $y$
into a new vertex~$v$.
Formally, we obtain a new graph $G'$ with
$V(G')=V(G)\cup\{v\}\setminus\{x,y\}$ and $E(G')=
\{\,ab\in
E(G): \{a,b\}\cap\{x,y\}=\emptyset\,\}\cup\{\,va: xa\in E(G)\text{ or }
ya\in E(G)\,\}$.  
If $G$ is a triangulation and we
contract the edge $xy\in E(G)$, then the resulting graph $G'$ is also
a triangulation provided that $xy$ is not part of a separating
triangle. 
Any plane drawing of $G$ leads naturally
to a plane drawing of $G'$.

\subsection{Characterization of Collinear Sets}




We will make use of the following strengthening of \thmref{collinear-set}
which follows from the proof in \cite{dalozzo.dujmovic.ea:drawing}:
\begin{thm}\thmlabel{dujmovic-frati}
  For any planar graph $G$, the following two statements are equivalent:
  \begin{compactenum}
    \item There is a plane drawing of $G$ and a proper good curve $C\colon
      [0,1]\to\R^2$ such that the sequence of edges and vertices
      intersected by $C$ is $r_1,\ldots,r_k$.
    \item There is a \Fary\ drawing of $G$ in which the sequence of
      edges and vertices intersected by the $y$-axis $Y$ is $r_1,\ldots,r_k$.
  \end{compactenum}
\end{thm}


\section{A-Graphs}
\seclabel{quad}
\seclabel{quadrangulations}

In this section, we study a special class of graphs that are closely
related to quadrangulations in which every edge crosses $Y$. (See \figref{a-graph} for an example.)

\begin{defn}\deflabel{a-graph}
	An \emph{A-graph}, $G$, is a plane straight-line graph with $n\ge 3$ vertices that has the following properties:
	\begin{compactenum}
		\item Every edge of $G$ intersects $Y$ in exactly one
                  point, possibly an endpoint.
                  \label{p1}
		\item Every face of $G$, including the outer face, is
                  a quadrilateral or a triangle (not containing any
                  disconnected components inside).
		\item Every quadrilateral face of $G$ is non-convex.
		\item Every triangular face contains one vertex
 on $Y$, one in $L$, and one in $R$.  
		\item Every vertex $v$ on $Y$ is incident to precisely
		two triangular faces, one ``above $v$'', which
                contains the line segment between $v$ and
                $v+(0,\epsilon)$ for some $\epsilon>0$, and one ``below
                $v$'', containing the line segment between $v$ and $v-(0,\epsilon)$ for some $\epsilon >0$.
                  \label{p-last}
	\end{compactenum}
\end{defn}

\begin{figure}
		\centering{\includegraphics[scale = 0.95]{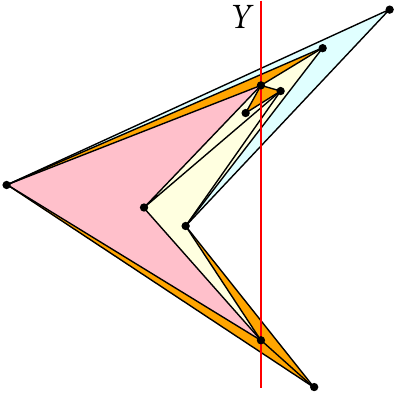}}
		\caption{An A-graph with 2 vertices on $Y$.}
		\figlabel{a-graph}
\end{figure}
	
In the special case where $G$ has no vertices in $Y$, the graph $G$ is a quadrangulation in which every edge crosses $Y$. Further, Property~5 applies even if $v$ is on the outer face of $G$ (in which case it implies that the outer face of $G$ must be a triangle).
Some additional properties of $G$ follow from 
Properties~\ref{p1}--\ref{p-last}.
\begin{compactenum}\setcounter{enumi}{5}
	\item $G$ is connected.
	\item Every vertex of $G$ has degree at least~2.   
	\item If $n\ge 4$, then every vertex in $Y$ has degree at least~3. 
\end{compactenum}
Property~6 follows directly from Property~2.
Property~7 follows from the fact that every vertex is incident to at
least one face and every face is a simple cycle.
Property~8 follows from the fact that every vertex on $Y$ is incident
to at least two triangular faces, which involve at least 4 vertices,
unless $n=3$.
(Property~3 and 4 are actually redundant---Property 3 follows from Properties~1 and 5; Property~4 follows from Property~1.)

%
%

In the following theorem,
we will show that every A-graph $G$ has a \Fary\ drawing with prescribed intersections with $Y$ and a prescribed outer face.

\begin{figure*}
   \begin{center}\begin{tabular}{ccc}
      \includegraphics{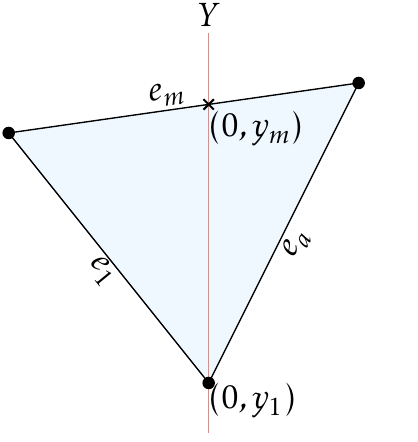} & 
      \includegraphics{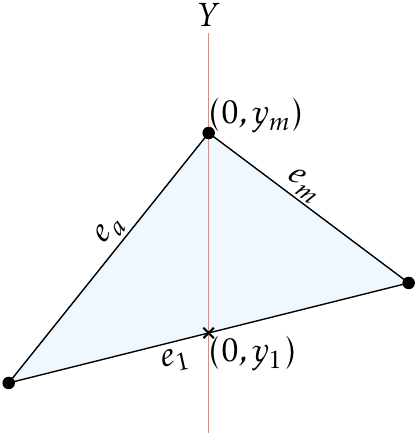} & 
      \includegraphics{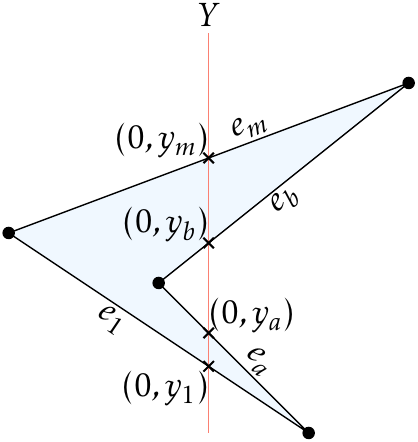} \\
      (a) & (b) & (c)
   \end{tabular}\end{center}
   \caption{The three possibilities for the outer face in
     \thmref{a-graph}.
   }
   \figlabel{outerface-cases}
\end{figure*} 

\begin{thm}\thmlabel{a-graph}\ 
\begin{compactitem}
\item Let $G$ be an A-graph.
\item Let $e_1,\ldots,e_m$ be the sequence of edges in $G$,
  in the order they are intersected by $Y$. Ties between edges having
  a common endpoint on $Y$ are broken arbitrarily,
except that $e_1$ and $e_m$ are
always edges on the outer face.
\item Let $y_1\le\cdots\le y_m$ be any sequence of numbers where, for
  each $i\in\{1,\ldots,m-1\}$, $y_i=y_{i+1}$ if and only if $e_i$
  and $e_{i+1}$ have a common endpoint in $Y$.
			
			
\end{compactitem}
Then $G$ has a \Fary\ drawing in which the intersection
between $e_i$ and $Y$ is the single point $(0,y_i)$, for each
$i\in\{1,\ldots,m\}$.

Moreover, the shape  $\Delta$ of the outer face can be prescribed,
subject only to the constraint that $\Delta$
has to be consistent with the graph and the data $y_1,\ldots,y_m$.
Specifically, we have three possibilities, which
 are illustrated in \figref{outerface-cases}.

\begin{compactenum}[a)]
\item If the outer face of $G$ is a triangle containing the lowest
  vertex on $Y$, then $\Delta$ must be a triangle with a vertex at
	$(0,y_1)$, and the opposite edge crosses $Y$ at $(0,y_m)$.
\item Symmetrically, if the outer face of $G$ is a triangle containing
  the highest vertex on $Y$, then $\Delta$ must be a triangle with a vertex
	at $(0,y_m)$, and the opposite edge crosses $Y$ at $(0,y_1)$.
\item Otherwise, the outer face of $G$ is a quadrilateral.  Let $e_1$,
  $e_a$, $e_b$, and $e_m$ be the 
  edges of the outer face.  Then $\Delta$ has to be a quadrilateral
	whose edges cross $Y$ at $(0,y_1)$, $(0,y_a)$, $(0,y_b)$, and $(0,y_m)$.
\end{compactenum}
\end{thm}

It would have been more natural to represent the intersection of $Y$ with $G$
as a mixed sequence of vertices and edges. However, to simplify the
statement of the theorem and its proof, we have chosen to specify the
desired drawing by a number $y_i$ for every edge, subject to equality
constraints. 
The convention in Condition~2 about $e_1$ and $e_m$ being boundary
edges is introduced only for
 notational convenience.

The rest of this section is devoted to proving \thmref{a-graph}. We
are going to prove \thmref{a-graph} in its strongest form, in which the
outer face $\Delta$ is prescribed.  We begin by making some simplifying
assumptions, all without loss of generality.
%
%
%
First, we assume w.l.o.g.\ that $\Delta$ and all vertices of $G$
are contained in the strip $[-1,1]\times(-\infty,+\infty)$.  This can
be achieved by a uniform scaling.  Second, if the outer face of $G$
is a quadrilateral, we assume w.l.o.g.\ that the common vertex
of $e_1$ and $e_m$ in the given drawing of $G$ is in $L$, as in
Figures~\ref{fig:a-graph} and \ref{fig:outerface-cases}c, and the
vertex of desired output shape $\Delta$ incident to $e_1$ and $e_m$
is also in $L$; this can be achieved by a reflection of $G$ or $\Delta$
with respect to $Y$.


If $m=3$ or $m=4$, then $G$ is a 3- or a 4-cycle, respectively, hence it suffices to draw it as $\Delta$. Therefore we assume, from now on, that $m\ge 5$.  


We will describe the desired \Fary\ drawing by assigning a slope $s_i$
to each edge $e_i\in E(G)$.   Since there can be no vertical edges,
each slope $s_i$ is well-defined. We have $m=|E(G)|$ slope variables,
$s_1,\ldots,s_m$. 
We can see that these variables determine the drawing:
 Since every edge $e_i$ contains the point $(0,y_i)$,
the slope $s_i$ fixes the line through $e_i$.  Since every vertex $v$
not on $Y$ is incident to at least two edges that contain distinct points
on $Y$, the location of $v$ is fixed by any two of $v$'s incident edges.  
(The location of each vertex on
$Y$ is fixed by definition.)  Our strategy is to construct a
system of $m$ linear equations in the $m$ variables $s_1,\ldots,s_m$,
and to show that
this system is feasible and that its solution gives the desired \Fary\
drawing of $G$.

A necessary condition for the slopes to determine a F\'ary drawing of
$G$ is that the 
edges 
with a common vertex should be concurrent. Let $v$ be a vertex 
not on $Y$, and let $e_i, e_j, e_k$ be three edges incident to $v$.
The fact that the supporting lines of $e_i$, $e_j$, and $e_k$
meet at a common point (the location of $v$) is expressed by the following
\emph{concurrency constraint} in terms of the slopes $s_i,s_j,s_k$:
\ifSODA
\begin{align}\eqlabel{slope0} 
\left|
\begin{matrix}
1&1&1\\
s_i&s_j&s_k\\
y_i&y_j&y_k
\end{matrix}
\right|&=
({y_j{-}y_k}) s_i + ({y_k{-}y_i}) s_j 
         + ({y_i{-}y_j})s_k  
                               = 0
\end{align}
\else
\begin{equation}\eqlabel{slope0} 
\left|
\begin{matrix}
1&1&1\\
s_i&s_j&s_k\\
y_i&y_j&y_k
\end{matrix}
\right|=
({y_j-y_k}) s_i + ({y_k-y_i}) s_j 
+ ({y_i-y_j})s_k  = 0
\end{equation}
\fi
Since $y_1,\ldots,y_m$ are given, this is a linear equation
in $s_1,\ldots,s_m$.
Writing this equation for all triplets of edges incident to a common
vertex $v$ will include many redundant equations. Indeed,
if $v$ has degree $d_v$,
it suffices to take $d_v-2$ equations: For each vertex $v\in V(G)$, we choose two fixed
incident edges $e_i$ and $e_j$ and run $e_k$ through the remaining
$d_v-2$ edges, specifying that $e_k$ should go through the common vertex
of $e_i$ and $e_j$.

Whenever convenient, we will use edges of $G$
as indices so that, if $e=e_i$ is an edge of $G$, then $s_e=s_i$
and $y_e=y_i$.  Further, if $e$ is a line segment that
intersects $Y$ in a point, we will use $y_e$ to denote the $y$-coordinate
of the intersection of $e$ and $Y$ and $s_e$ to denote the slope of~$e$. 


We now introduce additional equations for the edges that emanate from a
vertex on $Y$; refer to \figref{proportional}.
Suppose that a vertex $v\in Y$ is incident to edges $a_1,\ldots,a_k\in L\cup Y$ 
and $b_1,\ldots,b_\ell\in Y\cup R$, ordered from bottom to top as in \figref{ab}.

\begin{figure*}
  \begin{center}
    \includegraphics{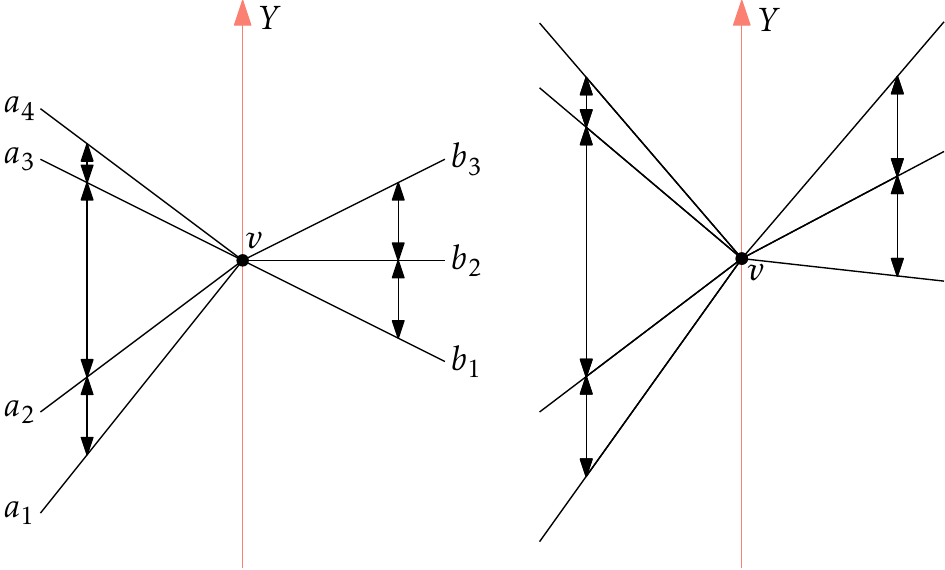}
  \end{center}
  \caption{The proportionality constraints on slopes of edges incident to a vertex $v\in Y$.}
  \figlabel{proportional}
\end{figure*}
From Property~4 of A-graphs we have $k,\ell\ge1$ and in addition
 $k+\ell\ge 3$ by Property~8.
Let us first look at the slopes on the right side.
We want these slopes to be increasing:
$s_{b_1} < s_{b_2} < \dots  <s_{b_\ell}$. We stipulate a stronger
condition:
We require that the slopes
$s_{b_2}, \dots, s_{b_{\ell-1}}$ partition the interval
$[s_{b_1},s_{b_\ell}]$ in fixed proportions. In other words:
\begin{equation}
\label{eq:proportion}
s_{b_i} = s_{b_1} + \lambda_i(s_{b_{\ell}}-s_{b_1}),
\end{equation}
for some fixed sequence $0<\lambda_2<\cdots<\lambda_{\ell-1}<1$.

For example, we might set $\lambda_i := (i-1)/(l-1)$.
This gives $\ell-2$ equations, for $\ell\ge 2$. Similarly, we get
$k-2$ equations for the slopes
$s_{a_1}, \dots, s_{a_{k}}$ of the edges on the left side, for $k\ge 2$.
In addition, for $k\ge 2$ and $\ell\ge 2$, we require that the \emph{range} of
slopes
on the two sides are in a fixed proportion:
\begin{equation}
\label{eq:proportion2}
s_{a_1}-s_{a_{k}} = \mu (s_{b_{\ell}}-s_{b_1}),
\end{equation}
for some fixed value $\mu>0$.

We call the equations
\thetag{\ref{eq:proportion}--\ref{eq:proportion2}} the
\emph{proportionality constraints}.
There are $(k+\ell)-3$ such equations for the $k+\ell$ slopes, hence
we have three degrees of freedom for the slopes incident to a vertex.
\figref{proportional} illustrates these  degrees of freedom:
Namely, we can shear the edges on the right side vertically, adding the same constant to all
slopes. We can independently shear all edges on the left side.
In addition, we can vertically scale {all} lines jointly (both to
the left and to the right), multiplying all slopes by the same constant factor.
If this factor is negative, we would reverse the order of the
slopes, simultaneously on the left and on the right. We will later see
that this undesirable possibility is prevented in conjunction with
other constraints that we are going to impose. We can already observe
that any two slopes on one side determine all remaining slopes on that side. Moreover, the range of slopes on the other side ($s_{a_1}-s_{a_{k}}$ or $s_{b_{\ell}}-s_{b_1}$) is also determined.
The notations $\lambda_i$ and $\mu$ are here used in a local sense;
for a different vertex $v$, we may choose different constants.
\begin{lem} \label{le:number-of-equations}
The total number of equations \thetag{\ref{eq:slope0}},~\thetag{\ref{eq:proportion}}, and~\thetag{\ref{eq:proportion2}} is $m-4$.
\end{lem} 
\begin{proof}
Let $n=|V|$ and let $n_0$ be the number
of vertices on $Y$. Assume that $G$ has $f_3$ triangular and $f_4$
quadrangular faces.

We have two triangles for every vertex on $Y$ (Properties 4 and 5 of A-graphs):
\begin{equation}
\label{eq:f3}
f_3 = 2n_0
\end{equation}
Euler's formula gives
\begin{equation}
\label{eq:Euler}
n + f_3+f_4 = m+2.
\end{equation}
Double-counting of edge-face incidences leads to the relation
\begin{equation}
\label{eq:edge-face}
3f_3+4f_4=2m.
\end{equation}
Denoting the degree of a vertex $v$ by $d_v$,
we have $d_v-3$ equations for each of the $n_0$ vertices $v$ on $Y$. For each of the 
$n-n_0$ vertices $v$ not on $Y$, 
we have $d_v-2$ equations.
The total number of equations is therefore
\ifSODA
\begin{align*}
P &= 
\sum_{v\in V\cap Y}(d_v-3)+
    \sum_{v\in V\cap(L\cup R)}(d_v-2)
  \\&
=
\sum_{v\in V}(d_v-2)-n_0
=
2m-2n-n_0.
\end{align*}
\else
\begin{align*}
P &= 
\sum_{v\in V\cap Y}(d_v-3)+
    \sum_{v\in V\cap(L\cup R)}(d_v-2)
=
\sum_{v\in V}(d_v-2)-n_0
=
2m-2n-n_0.
\end{align*}
\fi
Using \thetag{\ref{eq:f3}--\ref{eq:edge-face}}, this can be
simplified to
\ifSODA
\begin{align*}
  P&
      =
   2m-2n-n_0\\
 &
                 = 2m -2n -(2f_3+2f_4) +(2f_3+2f_4)-n_0\\
&= 2m -2(n +f_3+f_4) +\tfrac12(4f_3+4f_4-f_3)\\
 &= 2m -2(m+2) +\tfrac12(2m) = m-4.
   \qedhere
\end{align*}
\vspace{-2,2\baselineskip}
\hrule height 0pt
\else
\begin{align*}
  P
      =
   2m-2n-n_0
 &
                 = 2m -2n -(2f_3+2f_4) +(2f_3+2f_4)-n_0\\
&= 2m -2(n +f_3+f_4) +\tfrac12(4f_3+4f_4-f_3)\\
 &= 2m -2(m+2) +\tfrac12(2m) = m-4.
   \qedhere
\end{align*}
\fi
\end{proof}

To achieve the desired number $m$ of equations, we add
four \emph{boundary equations}.  If the outer face is a quadrilateral,
the desired slopes of the boundary edges already give us four equations:
We set
the slopes $s_1$, $s_a$, $s_b$, and $s_m$ of the boundary edges $e_1$, $e_a$,
$e_b$, and $e_m$ to the fixed values of the slopes of the edges of
$\Delta$.

If the outer face is a triangle, the shape $\Delta$ gives us
only three constraints for the slopes of the three edges $e_1$, $e_a$,
$e_m$.  If the triangle is $\alpha\beta\gamma$ with $\gamma\in Y$,
we arbitrarily pick another (non-boundary) edge $e_b$ incident to $\gamma$
and set its slope $s_b$ to an appropriate fixed value; this value has
to be either larger or smaller than each of $s_1$, $s_a$, and $s_m$
depending on whether $\gamma$ is the topmost or the bottommost point
on $Y$ and whether $e_b$ lies in $L$ or $R$. Together with the
proportionality constraints, this effectively pins \emph{all} slopes
incident to $\gamma$ to fixed values.

In both cases, we get 4 equations of the form
\begin{equation}
  \label{eq:boundary}
  s_i = h_i, 
\end{equation}
where $i\in\{1,a,b,m\}$.

Altogether, we now have a system of $m$ linear equations
in the $m$ unknowns $s=(s_1,\ldots,s_m)$, which we can write
compactly as
$A\cdot s = b$, with a square matrix $A$ whose entries come from
\thetag{\ref{eq:slope0}--\ref{eq:proportion2}}
and \eqref{eq:boundary}.
Only four entries of
the right-hand side vector
$b$
are non-zero, due to the four boundary equations.
We will show that $A\cdot s=b$ has a unique
solution and that this solution gives a \Fary\ drawing of $G$.

\subsection{Setting the Proportionality Constraints}
\label{sec:setting}


Our plan is to construct the desired drawing by a continuous morph, 
starting from the given drawing of $G$. Since the proportioonality 
constraints are not part of the output specification but were 
artificially added to achieve the right number of equations, we can make 
our life easy by just setting their coefficients so that they are 
satisfied by the initial drawing.
Specifically, the statement of \thmref{a-graph} assumes that $G$ is a
\Fary\ drawing.  In this drawing, every edge $e$
has a slope $s_e'$.
We use these slopes to set the
coefficients in the proportionality constraints.
Consider a
vertex $v\in Y$, incident to edges $a_1,\ldots,a_k$ and $b_1,\ldots,b_\ell$
as described above.
In the notation used
in \eqref{eq:proportion}, we set
\[
\lambda_i = (s_{b_i}'-s_{b_1}')/(s_{b_\ell}'-s_{b_1}') .
\]
The coefficients for the edges
 $a_1,\ldots,a_k$ on the left side are set similarly.
If $k\ge2$ and $l\ge 2$,
we set
\[
\mu = (s_{a_1}'-s_{a_k}')/(s_{b_\ell}'-s_{b_1}')
\]
 in \eqref{eq:proportion2}.
This ensures that the initial slopes $s_{1}',\ldots,s_{m}'$ satisfy the
proportionality constraints.

\subsection{Ordering constraints}

We define a relation $\prec$ on the edges of $G$, where $e_1 \prec e_2$ if and only if
\begin{itemize}
	\item $y_{e_1} < y_{e_2}$ and $e_1$ and $e_2$ have a common endpoint $v\in L$; or
	\item $y_{e_1} > y_{e_2}$ and $e_1$ and $e_2$ have a common endpoint $v\in R$.
\end{itemize}
We say that a vector $s=(s_1,\ldots,s_m)$ \emph{satisfies the ordering
	constraints} if $s_{e_1} < s_{e_2}$ for every pair $e_1,e_2\in E(G)$
such that $e_1\prec e_2$. This definition captures the condition that vertices of $G$ in $L$ (respectively, $R$) should be drawn so that they remain in $L$ (respectively, $R$), as in the following. 

\begin{obs}\obslabel{left-right}
If a solution $s$ to $A\cdot s=b$ satisfies the ordering constraints, then every vertex that is in $L$ (in $R$) in $G$ is also in $L$ (respectively in $R$) in the drawing corresponding to $s$. 
\end{obs}

\begin{proof}	  
Consider any vertex $v$ that is in $L$ in $G$ and that is incident to
(at least) two edges $e_1$ and $e_2$ with $y_{e_1} < y_{e_2}$, and hence $e_1 \prec e_2$. Since $s$ satisfies the ordering constraints we have $s_{e_1} < s_{e_2}$, hence the lines with slopes $s_{e_1}$ and $s_{e_2}$ through $(0,y_{e_1})$ and $(0,y_{e_2})$, respectively, meet in $L$. The argument for the vertices in $R$ is analogous. 
\end{proof}	  

By construction, the slopes $s_{e_1}',\ldots,s_{e_m}'$ of edges
in $G$ satisfy the ordering constraints, so the relation $\prec$ is acyclic.

\begin{lem}\lemlabel{order-gives-drawing}
	Any solution $s$ to $A\cdot s=b$ satisfying
	the ordering constraints 
	yields a
	\Fary\ drawing of $G$ with $\Delta$ as the outer face.
\end{lem}

\begin{proof}
	If $G$ is a plane drawing of a 2-connected graph, then another
	straight-line drawing $G'$ of the same graph $G$ is a \Fary\ drawing provided
	that two conditions are met:
	(i) For every vertex~$v$, the clockwise order of the
	edges around $v$ in $G'$ is the same as in $G$; and
	(ii) in the drawing $G'$, every face cycle of $G$ is drawn without crossings
	(Devillers, Liotta, Preparata, and Tamassia \cite[Lemma~16]{devillers.liotta.ea:checking}).
	
	In our case, $G'$ is a straight-line drawing of $G$ given by a solution
	to $A\cdot s = b$ that satisfies the ordering constraints.

	First we show that $G'$ satisfies condition (i).
        More specifically, we establish the following stronger
        property for every vertex $v$.
\begin{quote}
  \thetag{$*$}
The edges going to the right from $v$ are the same in $G$ and $G'$,
and their slopes have the same order in $G$ and $G'$.

The same properties hold for the edges to the left.
\end{quote}

        We distinguish the following cases:
	\begin{enumerate}
		\item $v\not\in Y$. Since $s$ satisfies the ordering
                  constraints, by \obsref{left-right} we know that
                  $v$ is on the same side ($L$ or $R$)
                  in $G$ and
                  in $G'$.
All incident edges go to one side.
                  This, together with the fact that the
                  orders in which the edges incident to $v$ intersect
                  $Y$ in $G$ and $G'$ agree implies that the
                  slope orders of the edges around $v$ in $G$ and $G'$ agree.
		
		\item
		$v\in Y$, with incident edges $a_1,\ldots,a_k\in
		L\cup Y$ and $b_1,\ldots,b_\ell\in Y\cup R$ as in
                \figref{ab}.
                Again, \obsref{left-right} ensures that
                these edges remain on the same side in $G'$.

\begin{enumerate}
\item If $v$ is a boundary vertex, 
  then the boundary equations fix the slopes of the two incident
  boundary edges, $a_1, b_1$ or $a_k, b_l$, plus a third edge.  As we
  already observed when the proportionality constraints were defined,
  these constraints then fix the slopes of all edges incident to $v$,
  so that their ordering agrees with that of $G$.
			
\item If $v$ is an interior vertex then,
  as discussed above, the proportionality constraints ensure that the
  slope order of $v$'s incident edges in $G'$ either matches that of
  $G$ on each side, or it is completely reversed on both sides.
  Let us assume for contradiction that the latter case happens:
  \begin{equation}
    \label{eq:not-ordered}
    s_{b_1}\ge s_{b_\ell}
    \text{ and }
    s_{a_k}\ge s_{a_1}
  \end{equation}
  Let $e$ be the third edge of the triangle with edges $a_1$ and
  $b_1$, and let $f$ be the third edge of the triangle with edges
  $a_k$ and $b_\ell$, see \figref{ab}.
			Then the ordering constraints for the endpoints of $e$ imply
			\begin{math}
			s_{b_1}<s_e<s_{a_1}
			\end{math},
			and the ordering constraints for the endpoints of $f$ imply
			\begin{math}
			s_{a_k}<s_f<s_{b_\ell}
			\end{math}.
			Together with \eqref{eq:not-ordered}, this leads to a contradiction.
		\end{enumerate}
\end{enumerate}

\begin{figure}
  \centering
  {\includegraphics[scale = 1]{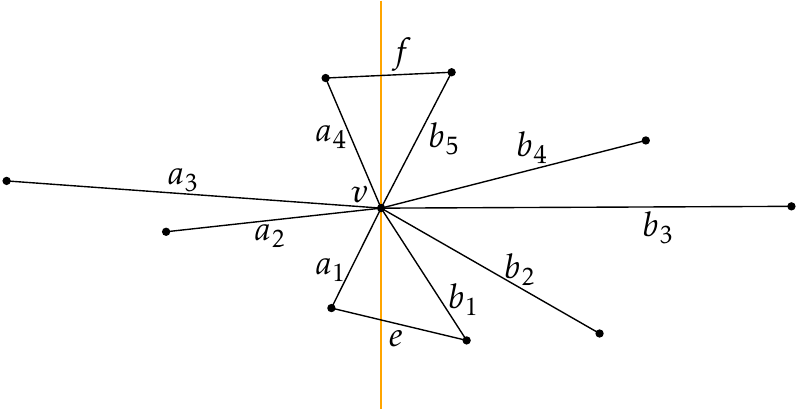}}
  \caption{The ordering of the edges incident to a vertex $v$ on $Y$.}
  \figlabel{ab}
\end{figure}

From the statement \thetag{$*$}, it is now easy to derive that
$G'$ satisfies condition (ii).
The graph $G$ has triangle and quadrilateral faces.
For a triangular face, \thetag{$*$} ensures that the triangle does not
degenerate, and is therefore non-crossing, in $G'$.
A quadrilateral face $q$ must be non-convex in $G$ by Property 3 of
A-graphs, and for each vertex, the two incident edges of~$q$ go in
the same direction (left or right).
Thus, Property \thetag{$*$} ensures that $q$ is non-crossing in~$G'$.


		
		
	Therefore, by the result of Devillers \etal\ cited above, $G'$
	is a \Fary\ drawing. That $G'$ has $\Delta$ as the outer
	face follows from the inclusion of the boundary equations in
	$A\cdot s = b$.
\end{proof}

Any solution $s$ to $A\cdot s=b$ has the outer face drawn as $\Delta$,
by the boundary equations, and the intersection between $e_i$ and $Y$ is $(0,y_i)$ by construction. Hence, by~\lemref{order-gives-drawing}, ensuring the existence of a solution $s$ to $A\cdot s=b$ satisfying the ordering constraints is enough to prove~\thmref{a-graph}.

\subsection{Strong Ordering Constraints}
\label{strong}

For some $\epsilon > 0$, we say that $s=(s_1,\ldots,s_m)$ satisfies
the \emph{$\epsilon$-strong ordering constraints} if, for each
$i,j\in\{1,\ldots,m\}$ such that $e_i\prec e_j$, the inequality
$s_j-s_i \ge \epsilon$ holds.
Clearly, any $s$ satisfying the $\epsilon$-strong ordering constraints
also satisfies the ordering constraints. 
The converse holds, for a suitably small $\epsilon$ (the inequalities
being strict in the definition of ordering constraints). The following
lemma tells us that this $\epsilon$ can be determined by $\Delta$ and by the
sequence $y_1,\ldots,y_m$.
%

\begin{lem}\lemlabel{weak-to-strong}
  If
  $\Delta\subset[-1,1]\times(-\infty,+\infty)$, then
	any solution $s$ to $A\cdot s=b$ that satisfies
	the ordering constraints
	also satisfies 
	the $\epsilon$-strong ordering constraints
	for all $\epsilon\le\min\{\,|y_i-y_j| : e_i\prec e_j\,\}$.
\end{lem}

\begin{proof}
	By \lemref{order-gives-drawing} every vertex is contained in
	 $\Delta$. 
	Hence, every $x$-coordinate is in the interval $[-1,1]$.
	If $e_i\prec e_j$, then the common vertex of $e_i$ and $e_j$ has $x$-coordinate
	$(y_j-y_i)/(s_i-s_j)$.	From $|(y_j-y_i)/(s_i-s_j)|\le 1$ we
	derive $|s_i-s_j|\ge|y_j-y_i| \ge \epsilon$.
\end{proof}

\subsection{Uniqueness of Solutions Satisfying Ordering Constraints}

\lemref{weak-to-strong} and the $\epsilon$-strong ordering
constraints play a crucial role in our proof because they
allow us to appeal to continuity: If the slopes change continuously,
it is impossible 
to violate
the ordering constraints without first violating the
$\epsilon$-strong ordering constraints.
But since the ordering constraints imply the
$\epsilon$-strong ordering constraints,
it is impossible
to violate
the ordering constraints at all.
An example of this argument will be seen in the following proof.

\begin{lem}\lemlabel{unique}
	If $s$ is a solution to $A\cdot s=b$ that satisfies the ordering
	constraints, 
	then $s$ is 
	the unique solution to $A\cdot s=b$.
\end{lem}

\begin{proof}
	Assume that $\epsilon$ is fixed so that $0<\epsilon\le\min\{\,|y_i-y_j| : e_i\prec e_j\,\}$.
	
	Suppose, for contradiction, that there is a solution $s$ to
        $A\cdot s=b$ that satisfies the ordering
        constraints, 
	but is not unique.  Since $A\cdot s=b$ is a linear system, it
        must then have a 1-parameter family of solutions $s+\lambda r$,
        $\lambda\in\mathbb R$, for some non-zero $m$-vector $r$.

	Define the continuous (in fact, piecewise linear) function
	\begin{equation*}
	f(\lambda) := \min \{\, (s_j+\lambda r_j)-(s_i+\lambda r_i) : e_i \prec
	e_j\,\}
	.
	\end{equation*}
 Let $\lambda^*$ be the value with the smallest absolute value
	$|\lambda^*|$ such that
	$f(\lambda^*)\le\epsilon/2$. In order to prove that
        $\lambda^*$ exists, it suffices to prove that $f(\lambda)\le
        0$ can be achieved. The vector $r=(r_1,\ldots,r_m)$ has at least four zero entries
	$r_1=r_a=r_b=r_m=0$ since the slopes $s_1$, $s_a$, $s_b$, and $s_m$
	are fixed.
	Since $G$ is connected and $m\geq 5$, there is at least one vertex $v$ with two incident edges $e_k$
	and $e_\ell$ such that $r_k=0$ and $r_\ell\neq 0$. 
	We can thus pick $\lambda$ so that $(s_\ell+\lambda r_\ell)-(s_k+\lambda r_k)=s_\ell-s_k+\lambda r_\ell=0$,
	and then $f(\lambda)\le 0$. It follows that $\lambda^*$ exists.
	
	Now we know that, for any $\lambda$ between $0$ and $\lambda^*$ and for any $i$ and $j$ such that $e_i\prec e_j$, the difference $(s_j+\lambda r_j)-(s_i+\lambda r_i)$ has the same sign as $s_j-s_i$. It follows that the slopes satisfy the ordering constraints throughout
	this interval, but then
	\lemref{weak-to-strong} implies that $f(\lambda^*)\ge\epsilon$, a contradiction.
\end{proof}


\subsection{A Parametric Family of Linear Systems}

We now define a parametric family of linear systems
 $A^t\cdot s = b^t$, parameterized by $0\le t\le 1$
 by varying the intersection points $y=(y_1,\ldots,y_m)$
 and the boundary slopes $h=(h_1,h_a,h_b,h_m)$.
 Let us first see how the coefficients $A$ and right-hand sides $b$ of the system change when these data are changed.
 The coefficients of the
 concurrency constraints
 \eqref{eq:slope0} depend linearly on $y$,
 whereas
 the
 proportionality constraints 
 \thetag{\ref{eq:proportion}--\ref{eq:proportion2}} remain unchanged,
 and the boundary constraints~\eqref{eq:boundary} have just the constant coefficient~1.
In the right-hand sides $b^t$,
the four nonzero entries
are the four slopes
$h=(h_1,h_a,h_b,h_m)$.

We derive the intermediate systems 
$A^t\cdot s = b^t$ by linear interpolation between the initial data
and the target data:
For the ``starting system'', we use
the intercepts
 $y^0=(y_1^0,\ldots,y_m^0)$
 and the slopes
 $h^0=(h_1^0,h_a^0,h_b^0,h_m^0)$
 of the edges in the initial drawing~$G$.
In the ``target system'', we use
the specified target intercepts
 $y^1=(y_1,\ldots,y_m)$
 and the slopes
 $h^1=(h_1,h_a,h_b,h_m)$ from the target shape~$\Delta$.
(If $\Delta$ is a triangle, this vector includes $h_b$ as an arbitrarily chosen
additional slope, as described earlier.)

We define the intermediate data $y^t$ and $h^t$
by linear interpolation:
\begin{equation*}
  y^t = (1-t)y^0 + ty^1,
  \qquad
  h^t = (1-t)h^0 + th^1.
\end{equation*}
This defines the corresponding intermediate systems 
$A^t\cdot s = b^t$, whose coefficients and right-hand sides depend
linearly on the parameter~$t$.

It is important to note that
the starting system
$A^0\cdot s = b^0$
has at least one solution, 
namely the slopes
$s^0=(s_1^0,\ldots,s_m^0)$
 of the edges in the initial drawing $G$.
 The proportionality constraints
  \thetag{\ref{eq:proportion}--\ref{eq:proportion2}} were designed
  in this way, as described in \secref{setting}.
 The
 concurrency constraints
 \eqref{eq:slope0} are fulfilled because the initial drawing $G$ is a
 straight-line drawing.
 The boundary constraints~\eqref{eq:boundary} are fulfilled by
 construction.

We will show that \lemref{weak-to-strong} can be applied to the
system $A^t\cdot s=b^t$, for every $0\le t\le 1$. We define
an appropriate threshold value $  \epsilon^*$ by
\begin{align*}
  \epsilon^*&
              =
              \min_{e_i\prec e_j}  
              \min_{0\le t\le 1}
              |y_j^t-y_i^t| 
 \\            &
              =
              \min_{e_i\prec e_j}  
              \min\{|y_j^0-y_i^0|,|y_j^1-y_i^1|\}
              > 0
\end{align*}

\begin{lem}
   For every $0\le t\le 1$, a solution $s$ to $A^t\cdot s = b^t$
   that satisfies the ordering constraints also satisfies the
   $\epsilon^*$-strong ordering constraints. 
\end{lem}

\begin{proof}
We denote by $\Delta^t$ the shape of the outer face as specified by
$y^t$ and $h^t$.  
It suffices to prove that this shape
is contained in
$[-1,1]\times(-\infty,+\infty)$,
at which point \lemref{weak-to-strong} applies.

We show that each vertex of $\Delta^t$ is in $[-1,1]\times[-\infty,+\infty]$.
If such a vertex $v$ does not lie on $Y$ and is incident to
the two outer edges
$e_i$ and $e_j$, with $i,j\in \{1,a,b,m\}$ and $e_i \prec e_j$, it
has $x$-coordinate $( y_i^t - y_j^t ) / ( s_j^t - 
s_i^t )$.
Consider the case that $v\in R$. So we want to show that
\begin{equation}
( y_i^t - y_j^t ) / ( s_j^t - s_i^t )  \le 1 \enspace . \eqlabel{pff}
\end{equation}
By the ordering constraints, $s_j^t - s_i^t > 0$, so \eqref{eq:pff} is
equivalent to
\begin{equation*}
  y_i^t - y_j^t  \le  s_j^t - s_i^t.
\end{equation*}
This inequality holds for $t=0$ and for $t=1$. The left side is linear
in $t$.
Since $e_i$ and $e_j$ are boundary edges, the right side is also
linear in $t$.
So the inequality holds for 
every $t\in [0,1]$.
In the case $v\in L$, the proof that $v$'s $x$-coordinate is at 
least $-1$ is similar.
\end{proof}

\subsection{Existence (and uniqueness) of solutions to $A^t\cdot s=b^t$}

We now prove the following lemma which, together with \lemref{order-gives-drawing}, completes the proof of \thmref{a-graph}.

\begin{lem}\lemlabel{uniqueness}
	For every $0\le t\le 1$, the system $A^t\cdot s=b^t$ has a
	unique solution $s^t$, and this solution satisfies the ordering
	constraints.
\end{lem}

\begin{proof}
	Since $A^t$ is an $m\times m$ matrix, the system $A^t\cdot
	s=b^t$ has a unique solution~$s^t$ if and only if $\det A^t \neq 0$.
	When $\det A^t =0$, the system may have no solutions or
	multiple solutions.  
	When $\det A^t\neq 0$, 
	Cramer's Rule states that
	the solution
	is $s^t=(s_1^t,\ldots,s_m^t)$ where, for each
	$i\in\{1,\ldots,m\}$,
	\[ 
	s_i^t = \frac{\det A^t_i}{\det A^t }
	\]
	and $A^t_i$ denotes the matrix $A^t$ with its $i$-th column replaced
	by $b^t$. 
	The numerators $\det A^t_i$ and the common
	denominator $\det A^t $ are polynomials in $t$, and therefore
	continuous
	functions of $t$.
	The solution $s^t=(s_1^t,\ldots,s_m^t)$ depends continuously on $t$
	as long as  $\det A^t\ne0 $.
	
	We have already established that $A^0\cdot s=b^0$ has a
	solution $s^0$ that satisfies the ordering constraints. By
	\lemref{unique}, this solution is unique, so $\det A^0\neq 0$.
	
	Let $t^*$ be the smallest $t>0$
	for which 
	$\det A^{t}= 0$. If such a value does not exist we set $t^*=2$.
	
	First we argue that, for all $0\le t <\min \{1,t^*\}$, the unique solution $s^t$ to $A^t\cdot s=b^t$ satisfies the ordering constraints. This argument is similar to the proof of \lemref{unique}. Suppose, for a contradiction, that there is a value $0<t<\min\{1,t^*\}$ for which $s^t$ does not satisfy the ordering constraints. As $t$ increases its value from $0$ to $\min\{1,t^*\}$, since $s^t$ depends continuously on $t$, a value is reached in which $s^t$ violates the $\epsilon^*$-strong ordering constraints, while it does not violate the ordering constraints. However, this contradicts	\lemref{weak-to-strong}.
	
	If $t^*>1$ the same argument also extends to $t=1$ and we are done.
	Let us therefore assume that $0<t^*\le 1$ and derive a contradiction.
	We look at the one-sided limit $s^*=\lim_{t\uparrow t^*}
	s^t$
as $t$ approaches $t^*$ from below.
	Each function $s_i^t$ is a quotient of two polynomials.
	Thus, for $t\to t^*$ it can either converge to $s_i^{t^*}$, or diverge to $+\infty$ or $-\infty$.
	For $t<t^*$ all solutions $s^t$ to the systems $A^t\cdot s=b^t$ satisfy the $\epsilon^*$-strong ordering constraints.
	Hence, if the limit exists, by continuity, it also satisfies $A^{t^*}\cdot s^*=b^{t^*}$
	and the $\epsilon^*$-strong ordering constraints.
	By \lemref{unique}, the solution $s^*$ is
	the unique solution
	of $A^{t^*}\cdot s=b^{t^*}$, but this contradicts the assumption
	that $\det A^{t^*}= 0$.
	
	It remains to rule out the possibility that
	$A^{t^*}\cdot s=b^{t^*}$ has no solution because
	$\lim_{t\uparrow t^*} s^t$ does not exist.  Define the set $H=\{\,e_i\in
	\{e_1,\ldots,e_m\}:\text{$\lim_{t\uparrow t^*} s_i^t$ exists}\,\}$.
	The set $H$ corresponds to the edges of $G$
	with bounded slope; the remaining edges become vertical as $t\to t^*$.
	\lemref{partition-extended} below shows that $H$ contains all edges of $G$. Hence $\lim_{t\uparrow t^*} s^t$ exists. This completes the proof of the lemma.
\end{proof}


It remains to prove that the set $H$ defined in the proof of \lemref{uniqueness} contains all edges in $E(G)$. We start by stating some properties of $H$.

	\begin{prop}\proplabel{set-H}
		The set   $H$ has the following properties: 
		\begin{compactenum}[(PR1)]
			\item $H$ contains every edge incident to a vertex on the outer face of $G$.
			\item \label{off-C}
			If a vertex $v\not\in Y$ has two incident edges in
			$H$,
			then all $v$'s incident edges belong to $H$.
			\item \label{on-C}
			If a vertex $v\in Y$ has two incident edges $vx,vy\in H$ with $x,y\in L$ or $x,y\in R$, then all $v$'s incident edges belong to $H$.
			\item If $e_i \prec e_j \prec e_k$ and $e_i,e_k\in H$, 
			then $e_j\in H$.
		\end{compactenum}
	\end{prop}
		
	\begin{proof}
 (PR1) If $v$ is a boundary vertex with $v\not\in Y$, then
the location of $v$ is fixed and the $y$-intercepts and therefore slopes of
$v$'s incident edges are fixed.  If $v$ is a boundary vertex with $v\in Y$ 
then $\Delta$ is a triangle and $v$ has three incident edges with slopes 
fixed by the boundary equations. Two of these edges are boundary edges, 
so two of these edges lie on the same side, say $L$, and the third edge 
lies on the other side, say $R$. By the proportionality constraints \eqref{eq:proportion}
all edges in $L$ are bounded, and thus belong to $H$. By the 
proportionality constraints \eqref{eq:proportion2} the range of slopes used by the edges in 
$R$ is bounded, and as one of them is fixed all of them have bounded 
slopes, and thus belong to $H$.

		(PR2)	If $v$ does not lie on $Y$ and two incident edges have bounded slope, then the location of $v$ is fixed in the limit.
		By the concurrency constraints, the slopes of the remaining incident
edges are also bounded.

(PR3) The case where $v$ lies on the outer face is subsumed by (PR1).
Assume therefore that $v$ lies on $Y$ and is an interior vertex of $G$.  Define the edges $a_1,\ldots,a_k$
		and $b_1,\ldots,b_\ell$ incident to $v$ as in \figref{ab}.  Let $e$ be the third edge of the triangle with edges $a_1$ and $b_1$, and let
		$f$ be the third edge of the triangle with edges $a_k$ and $b_\ell$.
		Assume without loss of generality that two of the edges $a_i$ belong to $H$. Then, by the proportionality constraints, all edges $a_i$ belong to $H$, and moreover the range $s_{b_\ell}^t-s_{b_1}^t$ converges to a bounded limit as $t\to t^*$.  It follows that either all slopes of the edges $b_j$ are bounded, or they all diverge to $+\infty$,
		or they all diverge to $-\infty$. The ordering constraints for the
		endpoints of $e$ imply \begin{math}
		s_{b_1}<s_e<s_{a_1}
		\end{math}.
		This is inconsistent with $\lim_{t\uparrow t^*} s_{b_1}^t=+\infty$.
		The ordering constraints for the endpoints of $f$ imply
		\begin{math}
		s_{a_k}<s_f<s_{b_\ell}
		\end{math}.
		This is inconsistent with $\lim_{t\uparrow t^*} s_{b_\ell}^t=-\infty$. Thus, the only
		possibility is that all slopes of the edges incident to $v$ are bounded.
		
		
(PR4) This follows from the ordering
constraints, since, for all $0\le t< t^*$,
		$s_i^t<s_j^t<s_k^t$ and both $\lim_{t\uparrow
		t^*}s_i^t$ and $\lim_{t\uparrow t^*}s_k^t$ are defined.
	\end{proof}
	
We now present \lemref{partition-extended}, which completes the proof of \lemref{uniqueness} and \thmref{a-graph}. The lemma is proved by induction on something
that starts as an $A$-graph but is then dismantled into something
more general.  A \emph{near-A-graph} is a graph that satisfies all
conditions of an A-graph except that its outer face can be arbitrarily
complex, even disconnected.  More specifically, each edge of a near-A-graph intersects $Y$
in exactly one point; each inner face is a triangle or a quadrilateral,
without any disconnected compoments inside;
each triangular face contains one vertex in each of $Y$, $L$, and $R$;
and for every vertex $v$ on $Y$ each of the faces directly above and
below $v$ is either a triangular face or the outer face.

\begin{lem}\lemlabel{partition-extended}
Let $G$ be a near-A-graph and let $H \subseteq E(G)$ be a set of edges satisfying Properties (PR1)--(PR4) of \propref{set-H}. 
	Then $H=E(G)$.
\end{lem}

\begin{proof}
	The proof is by induction primarily on the number of inner faces of $G$ and secondarily on the number of vertices of $G$. We dismantle $G$
	from outside while maintaining
	Properties~(PR1)--(PR4). In particular:
	\begin{itemize}
		\item If $G$ is not 2-connected but has more than one
                  edge, we will cut it
		into pieces with fewer edges.
		\item If $G$ is 2-connected, we will modify it and reduce it to a
		graph
		with fewer interior faces,
		keeping the number of edges fixed.
	\end{itemize}
Eventually, we will reduce to a graph with a single edge, and here the
claim is trivial because the edge belongs to the boundary.

We will now go into the details of the proof.
	We refer to the edges of $H$ simply as \emph{$H$-edges}.
	
	If $G$ is not connected then we can independently apply induction on each component
	of $G$. 
	
	If $G$ has a cut vertex $v$ whose removal
	splits $G$ into components $A_1,\ldots,A_r$ then, for each
	$i\in\{1,\ldots,r\}$, we can independently apply induction on the subgraph $G_i$ of $G$
	induced by $V(A_i)\cup\{v\}$. Every edge of $G_i$ inherits its
        classification as an $H$-edge from its corresponding edge in
        $G$.
        Then it is easy to see that Properties~(PR1)--(PR4) are
        satisfied by $G_i$.
        Properties~(PR2)--(PR4) are obviously preserved under taking subgraphs. 
        Property~(PR1) follows from the fact that every boundary vertex of $G_i$ is also a boundary vertex of $G$; this is because each inner face of $G$ is a quadrangle or a triangle, hence $G_i$ cannot be nested inside a different subgraph $G_j$ of $G$.

	\begin{figure*}[htb]
		\centering{\includegraphics{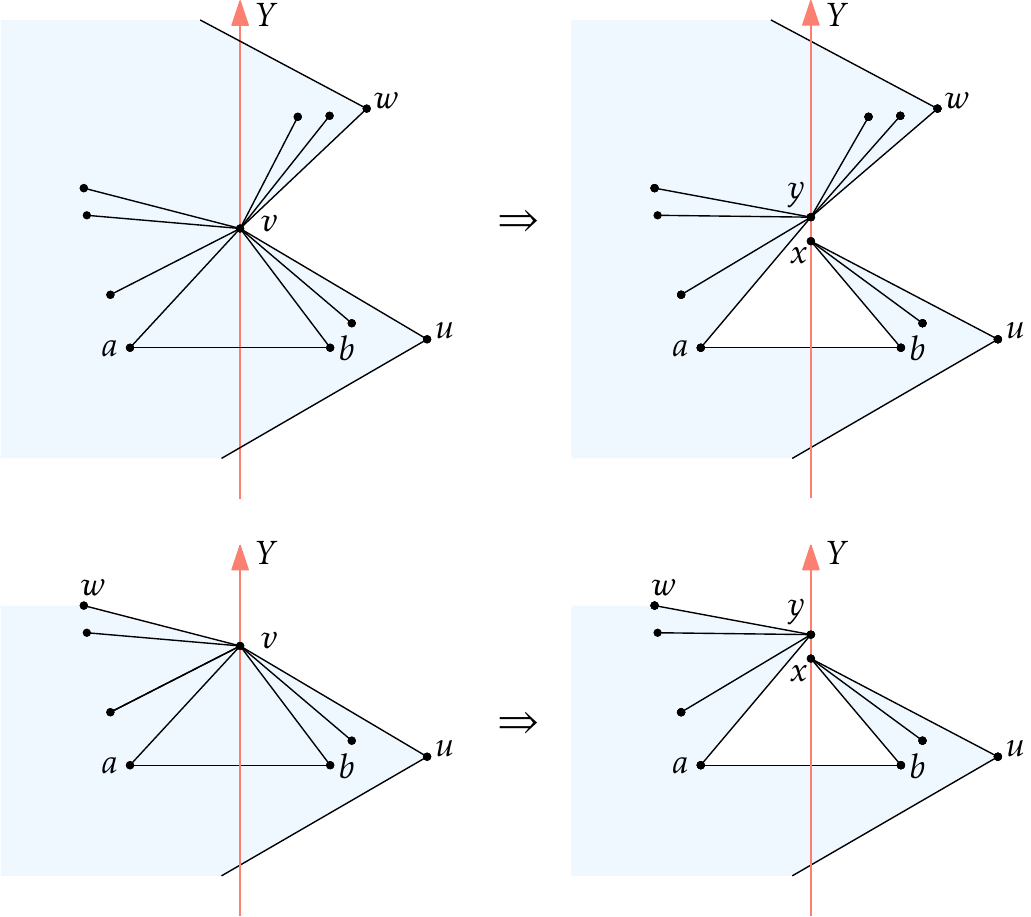}}
		\caption{Proof of \lemref{partition-extended} with a vertex
			$v$ on $Y$. Integrating a triangle into the outer face.}
		\figlabel{lemma-y-3}
	\end{figure*}

	We are left with the case in which $G$ is a 2-connected near-A-graph whose outer face 
	is delimited by a simple cycle $F$. We distinguish two cases.
	
	{\em Case 1}. The cycle $F$ contains a vertex $v$ on $Y$ that is incident to an inner triangular face $vab$. In this case
        we \emph{open up} $vab$, merging it into the outer face. \figref{lemma-y-3} illustrates the procedure for the case that $ab$ lies below $v$, with $a\in L$ and $b\in R$. Let $u$ and $w$ be the predecessor and successor of $v$ on the
	counterclockwise cycle $F$, and assume w.l.o.g.\ that $u\in R$.
	We construct a new graph $G'$ by splitting $v$ into two vertices $x$
	and $y$ that both lie on $Y$, with $y$ above $x$. We make $x$ adjacent to $u$ and to every neighbor of $v$ between $b$ and $u$.
	We make $y$ adjacent to the remaining neighbors of $v$.
	\figref{lemma-y-3} shows that this procedure works both for $w\in R$
	and for $w\in L$. Note that $G'$ has one inner face less than $G$, hence induction applies.

	{\em Case 2}. If Case~1 does not hold,
%
%
        every vertex of $F$ on $Y$ has all its neighbours
	in $L$ or has all its neighbours in $R$.  Since the same is already true for every vertex not on $Y$, a traversal of $F$ has to zigzag/alternate between edges that move to the left and edges that move to the right.  Thus, $F$ must contain some reflex vertex $v$.  The vertex $v$ cannot lie on $Y$, because otherwise we would be in Case~1.  Let $uv$ and $vw$ be the two consecutive edges of $F$ incident on $v$ and let $p$ and $q$ be the intersections of $uv$ and $vw$ with $Y$, with $q$ above $p$.	
	(Note that $u$ and/or $w$ may be contained in $Y$.)
	This implies that $v$ is a reflex vertex of some inner face
	$q=vabc$ of $G$.  Indeed, $vc$ is the first edge incident
	to $v$ intersected by $Y$ and $va$ is the last edge incident
	to $v$ intersected by $Y$. (Note the possibility that $a=w$
	and/or $c=u$.)	We construct a new graph $G'$ by splitting
	$v$ into two vertices $x$ and $y$. We make the vertex $x$
	adjacent to $u$ and every neighbour $z$ of $v$ such that $Y$
	intersects $vz$ before $vu$.  We make $y$ adjacent to all of
	$v$'s neighbors that are not adjacent to $x$.  \Figref{lemma-y-4}
	illustrates this procedure; the lower half illustrates the case
	where $w\in Y$. In $G'$, $q$ is part of the outer face, so $G'$
	has one less inner face than $G$, hence induction applies.

	\begin{figure*}[ht]
          \centering\includegraphics{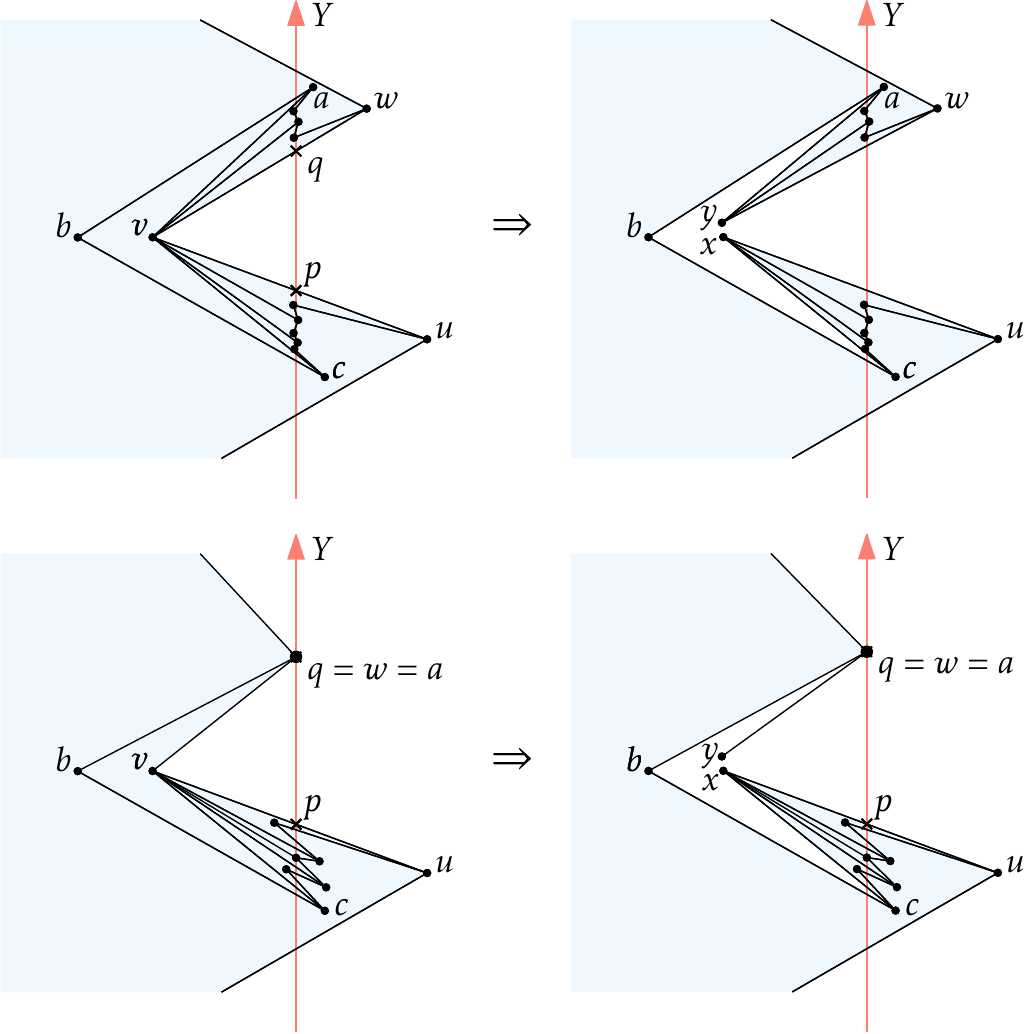}
          \caption{Proof of \lemref{partition-extended} for a reflex
            vertex $v$. Integrating a quadrilateral into the outer face.}
		\figlabel{lemma-y-4}
	\end{figure*}

	This finishes the description of how we modify $G$ into $G'$. Every edge of $G'$ inherits its classification as an $H$-edge from its corresponding edge in $G$. We have to show that $G'$ satisfies Properties~(PR1)--(PR4). Actually Property~(PR1) is the only property that needs to be discussed, as the other properties follow trivially from the fact that $G$ satisfies them. 
%
%
	
	First, note that all edges incident to the new vertices $x$ or $y$
	were incident to $v$ before, and thus they are $H$-edges. Second,
	all edges incident to any boundary vertex of $G'$ that is also
	a boundary vertex of $G$ are $H$-edges, since the edges of $G'$
	inherit their classification as $H$-edges from their corresponding
	edges in $G$. It remains to deal with the boundary vertices of
	$G'$ that are inner vertices of $G$.

	In Case 1 we have two boundary vertices of $G'$ that might be
	inner vertices of $G$, namely $a$ and $b$. These vertices
	do not lie on~$Y$. By Property~(PR1) for $G$, both $va$ and
	$vb$ are $H$-edges, since they are incident to the boundary
	vertex $v$. From the ordering constraints around $a$ and $b$
	we get $va\prec ab\prec vb$ or $vb\prec ab\prec va$, and thus,
	by Property~(PR4) for $G$, we have $ab\in H$.  Now we have two
	$H$-edges $va$ and $ab$ incident to $a$, and by Property~(PR2)
	for $G$ all edges incident to $a$ belong to $H$. It follows that
	all edges incident to $a$ in $G'$ are $H$-edges, and similarly
	for~$b$.

	In Case~2 we have three boundary vertices of $G'$ that might be inner
	vertices of $G$, namely $a$, $b$, and $c$. By Property~(PR1) for
	$G$, both $va$ and $vc$ are $H$-edges, since they are incident
	to the boundary vertex $v$. Consider the quadrilateral $q=vabc$
	of $G$. By the ordering constraints, we get $vc \prec bc\prec
	ba\prec va$ or $va \prec ba\prec bc\prec vc$, depending on
	whether $v\in L$ or $v\in R$. Thus, by Property~(PR4) for $G$,
	the edges $bc$ and $ba$ are also $H$-edges. The vertex $b$ does
	not lie on $Y$. The vertex $a$ might lie on $Y$ or not, but if it
	does, then the two incident edges $va$ and $ab$ lie in the same
	half-plane.  The same holds for $c$. Thus by Properties~(PR2)
	or~(PR3) for $G$ all edges incident to $a$, $b$ and $c$ in $G$
	belong to $H$. It follows that all edges incident to $a$, $b$,
	and $c$ in $G'$ are $H$-edges.

	Since $G'$ satisfies Properties~(PR1)--(PR4) induction applies and all edges of $G'$ (and thus all edges of $G$) are $H$-edges. This completes the proof.
\end{proof}


\section{Triangulations}
\seclabel{triangulations}



So far, we have shown that every collinear set in an A-graph is free,
and we can even specify  for edges that cross the line $Y$ the place
where this crossing occurs.
We will now apply this to prove for arbitrary planar graphs $G$
that every collinear set is free.
We might as well assume that $G$ is a maximal planar graph, i.e., a
triangulation.




\begin{thm}\thmlabel{main}\ 
  \begin{compactitem}
    \item Let $T$ be a triangulation, i.e.,
      a \textup(not necessarily straight-line\textup) plane drawing of an edge-maximal
      planar graph.
    \item Let $C$ be a good proper curve for $T$. \textup(This means that $C(0)=C(1)$
      is in the outer face and the intersection between $C$ and each edge
      $e$ of $T$ is either empty, a single point, or the entire edge $e$.\textup)
     \item Let $r_1,\ldots,r_k$ be the mixed sequence of vertices and open edges of $T$ that are intersected by~$C$, in the order in which they are intersected by $C$---edges of $T$ that lie entirely on $C$ are omitted from this sequence. \textup(They are implicitly represented by their endvertices, which are two consecutive elements $r_i$ and $r_{i+1}$.\textup)
	\item Let $y_1<\cdots<y_k$ be a sequence of numbers.
                  
       \item Let $\epsilon>0$ be a tolerance parameter.
\end{compactitem}
Then
        $T$
        has a \Fary\ drawing such that,
	for each $i\in\{1,\ldots,k\}$: 
	\begin{compactitem}
		\item if $r_i$ is a vertex, then it is drawn at $(0,y_i)$; and
		\item if $r_i$ is an edge, then the intersection of $r_i$ with $Y$ has $y$-coordinate in the interval $[y_i-\epsilon,y_i+\epsilon]$.
\end{compactitem}

Moreover, we can specify the shape $\Delta$ of the outer triangle,
subject to
obvious compatibility constraint that it intersects $Y$ in the specified points.
                
\end{thm}
The last condition can be formulated more explicitly:
The triangle $\Delta=\alpha\beta\gamma$ is \emph{compatible} with the
given data $r_1,\ldots,r_k$ and $y_1,\ldots,y_k$ if the following conditions hold:
\begin{compactitem}
	\item If $r_1$ is a vertex, then $\beta=(0,y_1)$, otherwise $(0, y_1)$ is in the interior of the edge $r_1=\beta\gamma$; and
	\item If $r_m$ is a vertex, then $\alpha=(0,y_m)$, otherwise $(0,y_m)$ is in the interior
	of the edge $r_m=\alpha\gamma$.
\end{compactitem}

If the tolerance $\epsilon$ is large, 
the statement of the theorem allows the order in which the edges cross
$Y$ to change. This is not intended, and it can be excluded if we choose
$\epsilon<\min\{(y_{i+1}-y_i)/2:i\in\{1,\ldots,k-1\}\}$. In the proof,
we will make this assumption.

\begin{figure*}[tb]
	\centering{\includegraphics{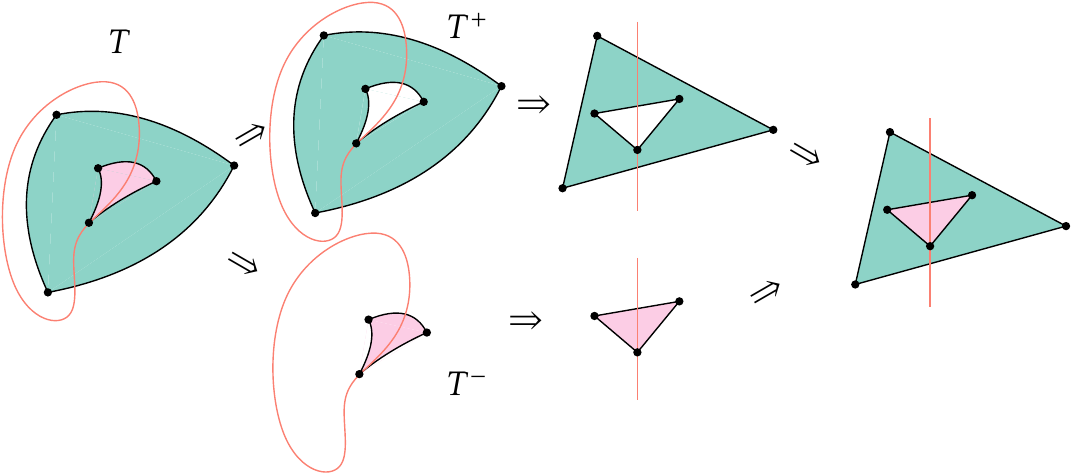}}
	\caption{Recursing on separating triangles in the proof of
		\thmref{main}}
	\figlabel{separating}
\end{figure*}

\begin{proof}
	%
  We start by classifying the edges of $T$.  An edge that has one
  endpoint in $C^-$ and the other endpoint in $C^+$ is a \emph{crossing
    edge}, otherwise it is a \emph{non-crossing edge}.
  An edge
  is \emph{marked} if it intersects $C$, otherwise it is
  \emph{unmarked}.  The unmarked edges are completely disjoint
  from~$C$.  The marked edges include all crossing edges, but also the
  edges with one endpoint on $C$ and the edges that lie completely on
  $C$.
	
	
	%

	The proof is a double induction on the number of vertices of $T$, primarily, and on the number of non-crossing edges of $T$, secondarily.
	We begin by describing reductions that allow us to apply the
	inductive hypothesis. When none of these reductions applies,
	we arrive at our base case. To handle the base case,
	we remove every unmarked edge of $T$, and we will show that we obtain an A-graph, to which we
	apply \thmref{a-graph}.

	

\paragraph{Separating Triangles.}
(See \figref{separating}.)
	If $T$ contains a separating triangle $xyz$, then denote by
        $T^+$ (respectively, $T^-$) the triangulation obtained from
        $T$ by removing the vertices in the interior
        (respectively, exterior) of $xyz$. The triangle that $xyz$ delimits an
        inner face of $T^+$ and the outer face of $T^-$.
Both $|V(T^+)|<|V(T)|$
and $|V(T^-)|<|V(T)|$, so we can apply induction if necessary.
        
The case that the interior of $xyz$ does not intersect $C$ is easy.
We draw $T^+$ by induction.  In this drawing, 
we take the triangle representing the cycle $xyz$, and we draw $T^-$
so that its outer face coincides with this triangle, for example by
Tutte's Convex Drawing Theorem~\cite{tutte:how}.

Consider now the case that $C$ intersects the interior of $xyz$. Then
$C$ intersects the boundary in two points: either it passes through a
vertex of $xyz$ and the opposite open edge, or it intersects two open
edges of $xyz$. 
In both cases, the vertices and edges of $T$ intersected by $C$ that
are not in $T^+$ form a nonempty contiguous subsequence
$r_i,\ldots,r_j$ of $r_1,\ldots,r_k$. Each of $r_{i-1}$ and $r_{j+1}$
is either an edge or a vertex of $xyz$.
	
Apply induction on $T^+$ with the value $\epsilon':=\epsilon/2$ and the sequences $r_1,\ldots,r_{i-1},r_{j+1},\ldots,r_k$ and
	$y_1,\ldots,y_{i-1},y_{j+1},\ldots,y_k$. In the obtained \Fary\ drawing of $T^+$ let $\Delta'$ be the triangle representing  $xyz$ and let $y_{i-1}'$ and $y_{j+1}'$
	be the respective $y$-coordinates of the intersections of
	$r_{i-1}$ and $r_{j+1}$ with $Y$.  By the choice of
	$\epsilon'$ we have $y_{i-1}'<y_i<\cdots<y_j<y_{j+1}'$.  Observe that
	$\Delta'$ is compatible with $r_{i-1},\ldots,r_{j+1}$ and
	$y_{i-1}',y_i,\ldots,y_j,y_{j+1}'$.
	We apply induction on $T^-$ with value~$\epsilon$ using the triangle $\Delta'$ and the sequences $r_{i-1},\ldots,r_{j+1}$ and
	$y_{i-1}',y_i,\ldots,y_{j},y_{j+1}'$.  Combining the \Fary\ drawings of $T^+$
	and $T^-$ yields the desired \Fary\ drawing of $T$.  Thus,
        from now on we assume that $T$ has no separating triangles.
	
	\paragraph{Contractible Edges.}
	(See \figref{contractible}.)
	A face of $T$ is a \emph{crossing
		face} if it is incident to two crossing edges. We declare an
	unmarked edge of $T$ to be \emph{contractible} if it is not contained
	in the boundary of any crossing face.  
	\begin{figure*}[htb]
		\centering{\includegraphics{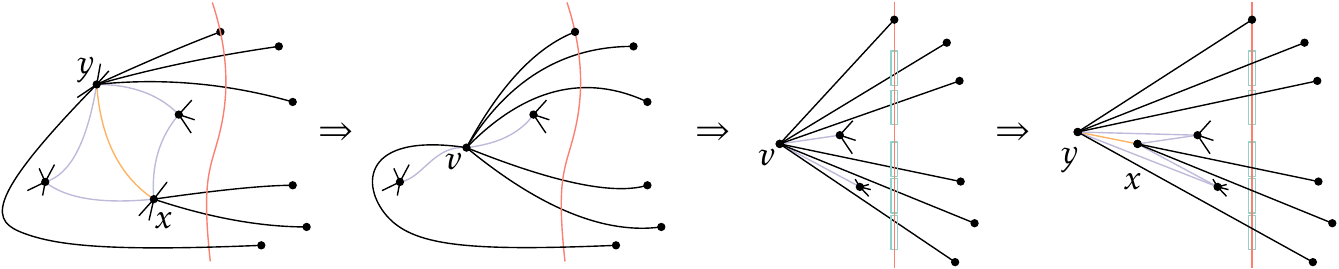}}
		\caption{Contracting and uncontracting edges in the proof of
			\thmref{main}}
		\figlabel{contractible}
	\end{figure*}
	
	If $T$ contains a contractible edge $xy$ then we contract $xy$ to
	obtain a new vertex $v$ in a smaller triangulation $T'$.   We then apply
	induction on $T'$ with the value $\epsilon'=\epsilon/2$ to obtain a \Fary\
	drawing of $T'$ such that each crossing edge $e_i$ crosses
	$Y$ in the interval $[y_i-\epsilon/2,y_i+\epsilon/2]$.
	
	To obtain a \Fary\ drawing of $T$ we uncontract $v$ by placing $x$ and $y$
	within a ball of radius $\epsilon/2$ centered at $v$. (Such
	a placement is always possible, by a standard argument, see, e.g.,~\cite{fary,w-sp-05}.)  Since the
	distance between $y$ and $v$ and the distance between $x$ and $v$ are each at most $\epsilon/2$,
	each crossing edge $r_i$ incident to $x$ or $y$ crosses $Y$ in the interval $[y_i-\epsilon,y_i+\epsilon]$, as required.
	Thus, in the following we assume that $T$ has no separating triangles or contractible
	edges.
	
	
\paragraph{Flippable edges.}
	(See \figref{flippable}.)
	We declare an unmarked edge $xy$ of $T$ to be \emph{flippable} if there
	exist distinct vertices $z$, $a$, $b$, and $c$ such that:
\begin{compactenum}[(11) ]
\item [(1)]
  $xyz$ is a non-crossing face of $T$;
\item [(2)]
  $xyb$, $zyc$, $xza$ are crossing faces of $T$;
  and either 
\item [(3a)] $C$
        intersects $za$, $xa$, $xb$, $yb$, $yc$, and $zc$ in this
        order, or
      \item [(3b)] $C$ intersects $xa$, $xb$, $yb$, $yc$, $zc$,
        and $za$ in this order. (This case can only occur when
        $xza$ is the outer face, otherwise $xza$ would be a separating
        triangle.)
        \end{compactenum}

	\begin{figure*}[htb]
		\centering{\includegraphics{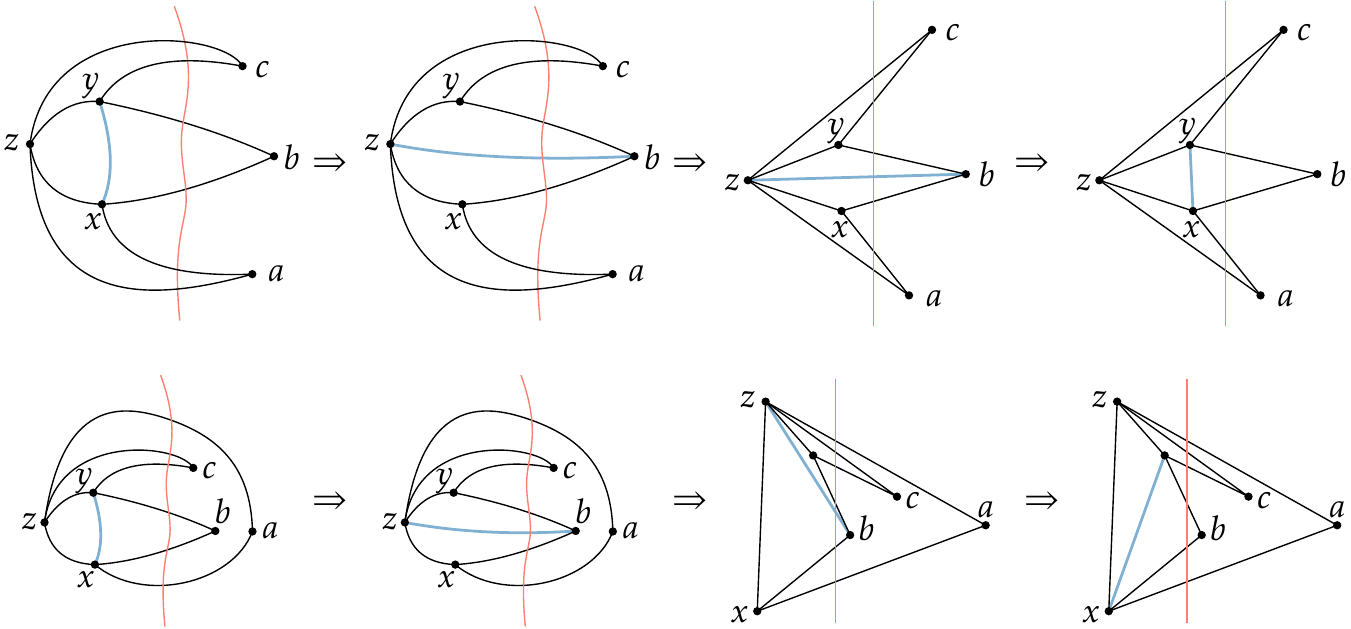}}
		\caption{Flipping edges in the proof of
			\thmref{main}. Top row: Case 3a. Bottom row:
                        Case~3b}
		\figlabel{flippable}
	\end{figure*}
	
	If $T$ contains the flippable edge $xy$ then we remove $xy$ and replace it with $zb$ to obtain a new triangulation $T'$. Note that, since $T$ has no separating triangles, the edge $zb$ is not already present in $T$. Further, $T'$ has the same number of vertices of $T$ and one less non-crossing edge. After choosing a crossing coordinate $y_{zb}$ for $zb$ between those $y_{xb}$ and $y_{yb}$ of $xb$ and $yb$, we can inductively draw $T'$ with tolerance $\epsilon$ and sequences $r_1,\dots,xb,zb,yb,\dots,r_k$ and $y_1,\dots,y_{xb},y_{zb},y_{yb},\dots,y_k$.
	
We claim that in the resulting \Fary\ drawing of $T'$,
we can replace $zb$ by $xy$ without creating a crossing, thus producing
 the desired \Fary\ drawing of $T$.
We show this by establishing that both $b$ and $z$ are convex vertices in $xbyz$.
The vertex $b$ is not a reflex vertex in $xbyz$, since $bx$ and $by$ are crossing
edges.
                In Case~(3a), the existence of the
	edges $za$ and $zc$ ensures that, in the \Fary\ drawing of $T'$,
	$xbyz$ is convex. In Case~(3b), the triangle $zxa$ is convex
        and $xbyz$ is contained in this triangle, therefore $z$ is
        a convex  vertex in $xbyz$.
	
	\paragraph{Edges on $C$.}
	
	If $T$ contains an edge $xy$ that lies on $C$, then we treat it as we treated flippable edges. In this case, $xy$ is incident to two triangles $xyz$ and $yxb$ with $z\in C^+$ and $b\in C^-$. We replace $xy$ with an edge $zb$ to obtain a new triangulation $T'$ with the same number of vertices of $T$ and one less non-crossing edge. We apply induction and get a \Fary\ drawing of $T'$, in which $z$ and $b$ are
	on opposite sides of $Y$ and $x$ and $y$ are on $Y$, hence 	neither $z$ nor $b$ is a reflex vertex of the quadrilateral $xzyb$.
	Thus, removing $zb$ and adding $xy$ gives a \Fary\ drawing of $T$.
	
	\paragraph{The Base Case.}
	
	We are left with the case in which $T$ is a triangulation
	with no separating triangles, no contractible edges, no flippable
	edges, and no edge contained in $C$.  If $T$ is the complete graph
	on three or four vertices, then the proof is trivial,
	so we may assume that $T$ has at least 5 vertices.

        We will simply omit the marked edges. The result will be an
        A-graph, to which we can apply \thmref{a-graph}. In the
        resulting drawing, we will see that we can reinsert the
        omitted edges without producing crossings.

	\begin{claimx} \label{unmarked}
	Any unmarked edge $xy$ in $C^+$ is on the boundary of two
        faces $xyz$ and $yxb$ where $z,b\in C \cup C^-$,
see \figref{two-triangles}a--b.
      \end{claimx}

      \begin{figure}[htb]
        \centering
        \includegraphics
        {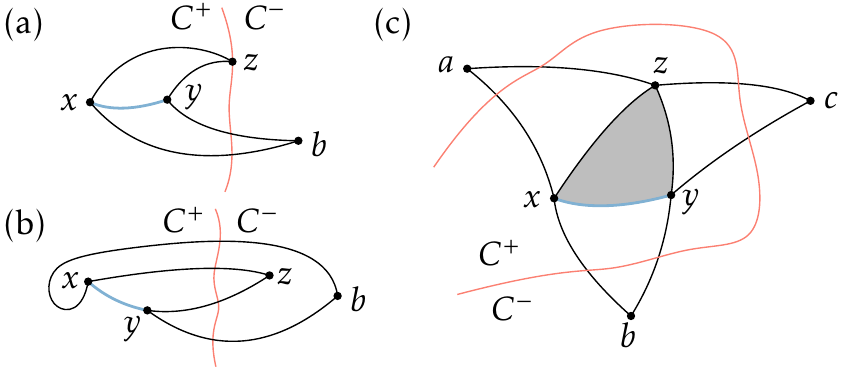}
        \caption{(a--b) shows how the two triangles incident to an unmarked edge
          $xy$ could look.
          In case (b), $yxb$ is the outer face.
          (c) The triangles adjacent to $xyz$ in the proof of Claim~\ref{unmarked}.
        }
        \label{fig:two-triangles}
      \end{figure}

	\begin{proof}
	Since $xy$ is not contractible, at least one of $xyz$ and $yxb$ is a
	crossing triangle, so at least one of $z$ and $b$, say $b$, is in $C^-$.
	Suppose then, for the sake of contradiction, that $z\in C^+$. Since neither $zx$ nor $yz$ is contractible,
	they must be incident to crossing faces $xza$ and $zyc$, respectively,
see \figref{two-triangles}c.
	If $a=b=c$, then $T$ is the complete graph on four vertices,
	which we have already ruled out.  Therefore, assume without loss of
	generality that $b\neq c$.  
	We have $a\neq c$, because otherwise $xya$ would
be a separating triangle that
        separates $z$ from $b$. Similarly, $a\neq b$, otherwise $byz$ would separate $x$ from $a$.	
	This leaves us in the situation in which we have distinct vertices $x$,
	$y$, $z$, $a$, $b$, and $c$ such that $xyz\in C^+$, such that $xyb$, $zyc$, and $xza$
	are crossing faces of $T$, and such that $xyz$ is a non-crossing face of $T$.
	Then at least one of $xy$, $yz$, or $zx$ is a flippable edge. This contradiction proves the claim.
\end{proof}	

	Symmetrically, every unmarked edge $xy$ in $C^-$ is incident to two faces
	$xyz$ and $yxb$ with $z,b\in C \cup C^+$.  This implies that no face of $T$ contains more than one unmarked edge.
	
	Thus, every unmarked edge of $T$ is incident to two faces that intersect $C$. The union of these two faces is a quadrilateral whose boundary consists of four edges that intersect $C$.
	Let $\tilde{G}$ denote the plane drawing obtained by removing all unmarked edges
	from $T$.  By \thmref{dujmovic-frati}, we know that $\tilde G$ has
	a \Fary\ drawing $G$ whose edges and vertices intersect $Y$ in the
	same order as they intersect $C$ in $\tilde{G}$. We have the following.
	
	\begin{claimx} \label{claim-a-graph}
		$G$ is an A-graph.
	\end{claimx}
	
	\begin{proof}
In order to	prove the claim, we check each of the properties of an A-graph
	(\defref{a-graph}).
	\begin{compactenum}
		\item The removal of unmarked edges and the fact that $T$ has no
		edge entirely on $C$ ensure that every edge of $G$ intersects $Y$
		in exactly one point.
		\item Because no face of $T$ is incident to more than one unmarked edge,
		each face of $G$ is a quadrilateral or a triangle.  
		\item A quadrilateral face $q=abcd$ appears in $G$ when we remove the unmarked edge $ac$ from $T$. This, and the fact that every edge of $q$ intersects $Y$, ensures that $a$ or $c$ is a reflex vertex of $q$.
		\item The only triangular faces of $G$ are those consisting
		of three marked edges, which necessarily have one vertex in
		each of $Y$, $L$, and $R$.
		\item Since $T$ has no edge on $C$, every vertex of $T$ on $C$
		is incident to two triangular faces (one above and one below)
		each having three marked edges. These faces are still present in $G$.
	\end{compactenum}
This concludes the proof of the claim.
\end{proof}	

	%
	%

We would now like to apply \thmref{a-graph} to obtain a \Fary\ drawing
of $G$ in which, for each $i\in\{1,\ldots,k\}$, the intersection of
$r_i$ with $Y$ is at $(0,y_i)$ and the appropriate vertices on the
outer face of $T$ map to the vertices of the triangle $\Delta$.
Before doing so, we must first prescribe an outer face $\Delta'$ for
the \Fary\ drawing of $G$.  If the outer face of $G$ is a 3-cycle,
then we use $\Delta'=\Delta$.  Otherwise, suppose the outer face of
$G$ is a 4-cycle $\alpha x\beta\gamma$ and $\alpha\beta$ is an
unmarked edge of $T$.  In this case, the locations of $\alpha$,
$\beta$, and $\gamma$ are given by the three vertices of $\Delta$
(with $\alpha$ and $\beta$ both on the same side of $Y$).
If $x$ lies on $Y$,
then $x=r_i$ for some~$i$, 
and the
position of $x$ is determined by $y_i$.
Otherwise, it is determined
by the positions of $\alpha$ and $\beta$ and the values
$y_{i}$ and $y_j$, where $r_i=\alpha x$ and $r_j=\beta x$.
\ifSODA\looseness-1\fi

	In this way, we can apply \thmref{a-graph} to
	obtain a \Fary\ drawing of $G$ in which the intersection of $r_i$
	with $Y$ is at $(0,y_i)$.  Each internal edge $ac$ of $T$ not in $G$
	corresponds to a quadrangular face $q=abcd$ of $G$ in which $a$ or $c$ is a
	reflex vertex.  Therefore, the edge $ac$ can be added to the drawing
	without introducing crossings.  A single external edge $\alpha\beta$
	on the outer face of $T$ might not appear in $G$. In this case the outer
	face of $G$ is a quadrilateral $q'=\alpha x \beta \gamma$ in which $x$ is
	a reflex vertex, so the segment $\alpha\beta$ lies outside of $q'$, and the edge $\alpha\beta$ can therefore be added to the drawing of $G$
	without introducing crossings. Therefore reinserting each edge of $T$ not in $G$ gives the desired \Fary\ drawing of $T$ (and the choice of $\Delta'$ ensures that the outer face of this drawing is $\Delta$). 
        This concludes the proof of \thmref{main}.
\end{proof}

We are finally ready to prove \thmref{our-bang}. Given a plane drawing of
a graph $G$, a collinear set $S$ in $G$, and any $y_1'<\cdots<y_{|S|}'$,
we need to prove that $G$ has a \Fary\ drawing in which the vertices in
$S$ are drawn at $(0,y_1'),\ldots,(0,y_{|S|}')$.  Let $C$ be the proper
good curve that contains $S$ and let $v_1,\ldots,v_{|S|}$ denote the
vertices of $S$ in the order they are encountered when traversing $C$
clockwise starting at the outer face.  

If $G$ is not a triangulation then we add edges to triangulate it in
such a way that each edge we add has a proper intersection with $C$.
To do this, we first add each edge $v_iv_{i+1}$ where $v_i$ and $v_{i+1}$
are in a common face of $G$ to obtain an augmented graph $G'$. If each
edge $v_iv_{i+1}$ added this way is drawn so that it coincides with
the subcurve of $C$ joining $v_i$ and $v_{i+1}$, then $C$ will be a
proper good curve for $G'$.  The interior of each face of $G'$ is either
entirely contained in the interior of $C$ or entirely contained in the
exterior of $C$. At this point we can greedily add edges to $G'$ until
it becomes a triangulation.  By \thmref{collinear-set}, the property
that $S$ is collinear set is  preserved.

 \thmref{dujmovic-frati} implies
	that there exists a Jordan curve $C$ that is admissible for $G$
	and that contains the vertices of $S$ in some order, say
	$v_1,\ldots,v_{|S|}$.  The curve $C$ intersects a subset of the edges
and vertices of $G$ in some order $r_1,\ldots,r_k$.
  We extend the sequence $y_i'$ by inserting additional elements,
  resulting in a
  sequence $y_1<\cdots<y_k$ so that,
  whenever $r_i=v_j$ for some $i\in\{1,\ldots,k\}$ and $j\in\{1,\ldots,|S|\}$,
  then
  $y_i = y_j'$.
  We select
	any triangle $\Delta$ that is compatible with $r_1,\ldots,r_k$ and
	$y_1,\ldots,y_k$ and choose $\epsilon = (1/3)\min\{\,y_{i+1}-y_{i}:
	1\le i \le k-1\,\}$.  \thmref{main} then gives us a \Fary\
	drawing of $G$ in which the vertices in $S$ are at
        $(0,y_1'),\ldots,(0,y_{|S|}')$, as required by
        \thmref{our-bang}.
Finally, edges that were inserted to create a triangulation are simply removed, and we obtain the desired \Fary\ drawing of the initial graph.
\qed


\section{Open Problems}

In this paper we proved that every collinear set is a free set. Several problems concerning collinear and free sets remain open. Here we mention our favorite two.

Let $f(n)$ be the minimum, over all $n$-vertex planar graphs $G$, of the size of the largest collinear set in $G$. What is the growth rate of $f(n)$? The best known bounds are $f(n)\in\Omega(\sqrt{n})$ and $f(n)\in \mathcal{O}(n^\sigma)$, for $\sigma < 0.986$ \cite{bose.dujmovic.ea:polynomial,ravsky.verbitsky:on}. Our results prove that $f(n)$ is also the minimum size of the largest free set over all $n$-vertex planar graphs; this makes determining the growth rate of $f(n)$ even more relevant. For example, any improvement in the lower bound would immediately give an improved result for untangling planar graphs.
%

We find it interesting to understand whether our main theorem,
\thmref{our-bang}, can be generalized so that the $y$-coordinates are arbitrarily prescribed not only for the vertices on~$Y$, but also for the crossing points of the edges with $Y$. Note that \thmref{main} {\em almost} gives this generalization, as every edge crossing $Y$ is at most $\epsilon$ away from its prescribed crossing point, for any arbitrarily small $\epsilon$. 

\section*{Acknowledgement}

Part of this research was conducted during the 5\textsuperscript{th} and the 6\textsuperscript{th} Workshops on Geometry and Graphs, held at the Bellairs Research Institute, March 5--10, 2017 and March 11--16, 2018.  We are grateful to the organizers and participants for providing a stimulating research environment.
%

\bibliographystyle{plain}
\bibliography{freecoll}

\end{document}